\documentclass[a4paper,UKenglish,cleveref,autoref, thm-restate]{lipics-v2021}

 \usepackage{xspace}
 \usepackage{bm}

\usepackage{multirow}
\usepackage[dvipsnames]{xcolor}
\definecolor{nicered}{RGB}{204,0,0}
\definecolor{lightblue}{RGB}{153,204,255}
\definecolor{nicegreen}{RGB}{0,153,0}
\usepackage{tikz}
\usetikzlibrary{decorations.pathreplacing, calligraphy}
\usetikzlibrary{decorations.pathmorphing}
\usetikzlibrary{arrows,automata}
\usetikzlibrary{arrows.meta}
\usetikzlibrary{shapes.geometric}
\usetikzlibrary{calc}
\usetikzlibrary{fpu}
\usetikzlibrary{math}

\title{Finding Minimum Matching Cuts in $H$-free Graphs} 

\titlerunning{Finding Minimum Matching Cuts} 

\author{Felicia Lucke}{LIP, ENS Lyon, France}{felicia.lucke@ens-lyon.fr}{https://orcid.org/0000-0002-9860-2928}{Supported by the EPSRC
(Grant No.\ EP/X01357X/1) and Postdoc.Mobility Grant 230578.}

\author{Joseph Marchand}{ENS Paris-Saclay, France}{}{}{}

\author{Jannik Olbrich}{Ulm University, Germany}{jannik.olbrich@uni-ulm.de}{https://orcid.org/0000-0003-3291-7342}{Supported by the Deutsche Forschungsgemeinschaft (DFG) (Grant No.\ OH 53/7-1).}

\authorrunning{F.~Lucke, J.~Marchand and J.~Olbrich}

\Copyright{Jane Open Access and Joan R. Public} 



\category{} 

\relatedversion{} 

\acknowledgements{The authors thank Van Bang Le for his contributions to a preliminary version of Theorem~\ref{t-s113} that got generalised afterwards. 
The authors further thank Clément Dallard, Daniel Paulusma, and Bernard Ries for fruitful discussions.}

\nolinenumbers 

\EventEditors{John Q. Open and Joan R. Access}
\EventNoEds{2}
\EventLongTitle{42nd Conference on Very Important Topics (CVIT 2016)}
\EventShortTitle{CVIT 2016}
\EventAcronym{CVIT}
\EventYear{2016}
\EventDate{December 24--27, 2016}
\EventLocation{Little Whinging, United Kingdom}
\EventLogo{}
\SeriesVolume{42}
\ArticleNo{23}

\tikzstyle{vertex}=[thin,circle,inner sep=0.cm, minimum size=1.7mm, fill=black, draw=black]
 \tikzstyle{svertex}=[thin,circle,inner sep=0.cm, minimum size=1.3mm, fill=black, draw=black]
 \tikzstyle{bvertex}=[thin,circle,inner sep=0.cm, minimum size=1.7mm, fill=lightblue, draw=lightblue]
 \tikzstyle{rvertex}=[thin,circle,inner sep=0.cm, minimum size=1.7mm, fill=nicered,draw=nicered]
 \tikzstyle{evertex}=[thin,circle,inner sep=0.cm, minimum size=1.7mm, fill=none,draw=black]
 \tikzstyle{edge}=[thick, draw = gray]
 \tikzstyle{tedge}=[ultra thick, draw = black]
 \tikzstyle{tredge}=[ultra thick, draw = nicered]
 \tikzstyle{tbedge}=[ultra thick, draw=lightblue]
 \tikzstyle{redge}=[thick, draw = nicered]
 \tikzstyle{bedge}=[thick, draw = lightblue] 
 \tikzstyle{gedge}=[thick, draw = nicegreen] 
 \tikzstyle{brace} = [decorate, ultra thick, decoration = {calligraphic brace}]
 \tikzstyle{wiggly} = [decorate, decoration = snake, thick, draw = gray,]

\newcommand{\NP}{\textsf{NP}}
\newcommand{\p}{\textsf{P}}

\newcommand{\mc}{{\sc Matching Cut}}
\newcommand{\pmc}{{\sc Perfect Matching Cut}}
\newcommand{\dpm}{{\sc Disconnected Perfect Matching}}
\newcommand{\maxmc}{{\sc Maximum Matching Cut}}
\newcommand{\minmc}{{\sc Minimum Matching Cut}}

\newcommand{\dc}{{\sc $d$-Cut}}
\newcommand{\matchmultc}{{\sc Matching Multi-Cut}}

\newcommand{\maxcut}{{\sc Max Cut}}

\newcommand{\mincut}{{\sc Min Cut}}

\newcommand{\degree}{\mathrm{degree}\xspace}
\newcommand{\dist}{\mathrm{dist}\xspace}
\newcommand{\diam}{\mathrm{diam}\xspace}

\theoremstyle{claimstyle}
\newtheorem{myclaim}{Claim}[theorem]

\newcommand{\clapp}[2]{{ \medskip \noindent $\vartriangleright$ {\sffamily Claim #1.{#2}.}}}

\newcommand\blfootnote[1]{%
    \bgroup
    \renewcommand\thefootnote{\fnsymbol{footnote}}%
    \renewcommand\thempfootnote{\fnsymbol{mpfootnote}}%
    \footnotetext[0]{#1}%
    \egroup
}

\begin{document}
\maketitle              
\begin{abstract}
A matching cut is a matching that is also an edge cut. In the problem \minmc{}, we ask for a matching cut with the minimum number of edges in the matching. We investigate the differences in complexity between \minmc, its counterpart \maxmc, and the decision problem \mc. Our polynomial-time algorithms for $P_8$-free, $S_{1,1,3}$-free and $(P_6 + P_4)$-free graphs extend the cases where \minmc{} and \maxmc{} are known to differ in complexity. In addition, they solve open cases for the well-studied problem \mc{}. The \NP-hardness proof for $3P_3$-free graphs implies that \minmc{} and \mc{}, which is polynomial-time solvable even for $sP_3$-free graphs, for any $s \geq 1$, differ in complexity on certain graph classes. Further, we give complexity dichotomies for both general and bipartite graphs of bounded radius and diameter.

\end{abstract}

\section{Introduction}
\label{sec:intro}

Given a graph $G$, a \emph{matching} is a set of disjoint edges. For a partition of the vertex set $V(G) = A \cup B$, the set of edges with one endvertex in $A$ and one in $B$ is an \emph{edge cut} of $G$. A \emph{matching cut} $M$ is a set of edges which is a matching and an edge cut. The number of edges $|M|$ is the \emph{size} of a matching cut.

The problem \mc, where we ask whether a given graph has a matching cut dates back to 1970 when Graham introduced it to prove a result on cube numbering~(\cite{Gr70}). Both polynomial-time algorithms and \NP-completeness results have been published for many different graph classes such as chordal graphs~(\cite{Mo89}), graphs of bounded diameter~(\cite{BJ08}), graphs of bounded degree~(\cite{Ch84}), and graphs of large girth~(\cite{FLPR25}).
Further, there are FPT algorithms~(\cite{AKK22,GS21,KKL20}), enumeration results~(\cite{GKKL22,GJMS24}), and exact algorithms~(\cite{CHLLP21,KL16,LT22,MM23}).
Different variants and generalisations of the problem have been studied. Among them are \pmc~(\cite{LT22}), \dpm~(\cite{BP25}), \dc~(\cite{GS21}), and the recent \matchmultc~(\cite{GJMS24}). Also, the forbidden subgraph decompositions considered in~\cite{RS02} are a generalisation of matching cuts.

In this work we consider those variants that take into account the size of a matching cut.
The oldest such variant is \pmc, introduced in~\cite{HT98}. A \emph{perfect} matching is a matching such that every vertex is incident to an edge in the matching. For \pmc, we ask whether in a graph $G$, there exists a matching cut that is a perfect matching, see in particular~\cite{BCD24,LT22}.
A \emph{maximum} (resp.~\emph{minimum}) matching cut is a matching cut with the maximum (resp.~\emph{minimum}) number of edges.
Only recently, in~\cite{LPR24}, the optimisation problem \maxmc{} was introduced.
\minmc, its natural counterpart, was introduced afterwards in~\cite{LLPR24}.
Observe that if ({\sc Perfect}) \mc{} is \NP-complete for some graph class, then the optimisation problem \maxmc{} is \NP-hard for that same graph class.
Similarly, \NP-completeness results for \mc{} imply \NP-hardness for \minmc.
For polynomial-time results we obtain implications in the other directions.

\begin{table}[]
    \centering
\begin{tabular}{|c | cccc| }

\hline
\rule{0pt}{3.5mm}
&   & \textsc{Minimum} & \textsc{Maximum}  & \textsc{Perfect} \\
& \textsc{Matching Cut} & \textsc{Matching Cut}  &  \textsc{Matching Cut} &  \textsc{Matching Cut}  \\[1mm]
\hline
\rule{0pt}{3.5mm}

\multirow{3}*{\p} & $sP_{3} + \bm{P_{8}} $ \cite{LPR22}   & $\bm{sP_2 + P_8}$ &   $ sP_{2} + P_{6}$ \cite{LPR24}  & $sP_4 + P_6$ \cite{LPR23a}\\
&  $sP_3 + \bm{P_4 + P_6}$ \cite{LPR22} & $\bm{sP_2 + P_4 + P_6}$ &  & \\[1mm]
& $sP_3 + \bm{S_{1,1,3}}$ \cite{LPR22} & $\bm{sP_2 + S_{1,1,3}}$ & & $sP_4 + S_{1,2,2}$ \cite{LT22,LPR23a} \\[2mm]
\hline
\rule{0pt}{3.5mm}
\multirow{6}*{\NP} & $3P_5$ \cite{LPR23a}  & $\bm{3P_3}$    & $2P_3$ \cite{LPR24}  & $3P_6$ \cite{LL26} \\
& $2P_7$ \cite{LL26}  & ($P_{11}$)     & ($P_{7}$)  & $2P_7$ \cite{LL26} \\
& $P_{14}$ \cite{LL26}   & & & $P_{14}$ \cite{LL26}\\[1mm]
& $K_{1,4}$ \cite{Ch84} & $(K_{1,4})$ & $K_{1,3}$ \cite{LPR24} & $K_{1,4}$ \cite{LT22}\\
& $H_i^*, i \geq 1$ \cite{FLPR25} & $(H_i^*, i \geq 1)$ & ($H_i^*, i \geq 1$) & $H_i^*, i \geq 1$ \cite{LT22}  \\
& $C_r, r \geq 3$ \cite{FLPR25} & $(C_r, r \geq 3)$ & $(C_r, r \geq 3)$ & $C_r, r \geq 3$ \cite{LT22} \\[1mm]
\hline

\end{tabular}
    \caption{The complexity of different \mc{} variants on $H$-free graphs. Results in bold are shown in this paper. Results in parentheses are direct implications of other results in the table.}
    \label{tab:overview}
\end{table}

In Table~\ref{tab:overview}, we give an overview on results for \mc, {\sc Minimum, Maximum} and \pmc{} on $H$-free graphs, which received particular attention in recent years. All graphs mentioned in the table are defined in Section~\ref{sec:prelim}.
In~\cite{FLPR25,LT22} it was proven that ({\sc Perfect}) \mc{} remains \NP-complete on graphs of fixed arbitrary girth and thus for $C_r$-free graphs, for $r \geq 3$. Further, it is known since 1984 that \mc{} is \NP-complete for $K_{1,4}$-free graphs~(\cite{Ch84}). A similar result exists for {\sc Perfect} and \maxmc~(\cite{LT22,LPR24}). This implies that all open cases are for $H$ being a forest where each component is a path or a subdivided claw.

Hence, there is a special focus on $P_r$-free graphs, where $P_r$ is the path on $r$ vertices. The first explicit result, a polynomial-time algorithm for \mc{} on $P_4$-free graphs, dates back to 2009 (\cite{Bo09}). In~\cite{Fe23}, Feghali showed polynomial-time solvability for $P_5$-free graphs, which got extended to $P_6$-free graphs in~\cite{LPR22}. A polynomial-time algorithm for \pmc{} on $P_6$-free graphs was given in~\cite{LPR23a}. Complementary \NP-completeness results have been published in~\cite{Fe23,LPR22} for \mc. 
The currently best result shows \NP-completeness for $P_{14}$-free graphs~(\cite{LL26}) for both \mc{} and \pmc. This leaves the case of $P_r$-free graphs, for $7\leq r \leq 13$, open.

Interestingly, \maxmc{} is the only matching cut variant for which a complexity dichotomy for $H$-free graphs is known so far (see~\cite{LPR24} and Table~\ref{tab:overview}). This is mainly due to the fact that it is \NP-hard for $K_{1,3}$-free graphs and $P_7$-free graphs. The other variants are either polynomial-time solvable or the complexity is open for these graphs. 
Prior to this work, \maxmc{} was the only variant for which the complexity on $P_7$-free graphs was known. 
For \minmc, a polynomial-time algorithm for $K_{1,3}$-free graphs was given in~\cite{LLPR24}, but a systematic complexity analysis was missing so far.

We initiate this analysis and in particular contribute to closing the gap for $P_r$-free graphs by giving a polynomial-time algorithm for $P_8$-free graphs for \minmc{} (Theorem~\ref{t-p8}). Note that this implies the same result for \mc.
This raises the question whether there is a graph $H$ such that \minmc{} and \mc{} differ in complexity on $H$-free graphs. We answer this question affirmatively by showing the \NP-hardness of \minmc{} for $3P_3$-free graphs (Theorem~\ref{t-3p3}), on which \mc{} is polynomial-time solvable~(\cite{LPR22}).

Another interesting case is that of $K_{1,3}$-free graphs, for which Bonsma (\cite{Bo09}) showed equivalence between \mc{} and {\sc Edge Cut}. In~\cite{LLPR24,LPR24} this was used to give a polynomial-time algorithm for \minmc{} using \mincut, and to show \NP-hardness for \maxmc{} using \maxcut. 
This revealed the first graph~$H$ for which \maxmc{} and \minmc{} differ in complexity for $H$-free graphs.
We extend this approach further, and give a polynomial-time algorithm for \minmc{} on $S_{1,1,3}$-free graphs using \mincut.
We further contribute to the general aim of completing the complexity dichotomy for ({\sc Minimum}) \mc{} for $H$-free graphs with our polynomial-time algorithm for $(P_6 + P_4)$-free graphs.

In addition to $H$-free graphs, the complexity on (bipartite) graphs of bounded radius and diameter was investigated for most known variants. We observe that our hardness construction has radius~$2$ and diameter~$3$. Further, we present a modification to obtain a bipartite graph of radius~$3$ and diameter~$4$. Together with the observation that some results for \maxmc{} can be easily modified to hold for \minmc, we obtain a complexity dichotomy for (bipartite) graphs of bounded radius and diameter.

\section{Preliminaries}
\label{sec:prelim}
We only consider finite, simple, undirected graphs. Let $G = (V,E)$ be a graph and $v \in V$ a vertex. We denote by $n$ the number of vertices of $G$. The set $N(v) = \{u \in V| uv \in E\}$ is the \emph{neighbourhood} of $v$ and $|N(v)|$ the \emph{degree} of $v$. We say that a vertex $u\in V$ is a \emph{private neighbour} of $v$ (with respect to a set $S$) if it is only adjacent to $v$ (in $S$). Let $C\subseteq V$. We say $v$ is adjacent to $C$ if there is a vertex $u \in C$ such that $v$ is adjacent to $u$. We denote by $\dist_C(u,v)$ the length of a shortest path between $u$ and $v$ using only vertices in $C$. 

We denote by $G[C]$ the graph induced by $C$. Let $H$ be a graph. If $H$ is an induced subgraph of $G$, we write $H \subseteq_i G$ or $ G \supseteq_i H$. We say that a graph is \emph{$H$-free} if it does not contain $H$ as an induced subgraph. We denote by $sG$, $s \geq 1$, the graph arising from the disjoint union of $s$ copies of $G$.
Let $r, \ell \geq 1$ be integers. Let $P_r$, $C_r$, $K_r$ be the path, cycle, complete graph, each on $r$ vertices.
A graph is \emph{bipartite} if there is a partition $V = A \cup B$ such that all edges of $G$ have one endvertex in $A$ and one in $B$. We denote by $K_{r,\ell}$ the \emph{complete bipartite graph} with $|A| = r$ and $|B| = \ell$. We call the graph $K_{1,3} = S_{1,1,1}$ a \emph{claw}.
$S_{i,j,k}$, for $i,j,k \geq 1$, can be obtained by subdividing the edges of a claw $i-1,j-1$ and $k-1$ times, respectively.
Let $H^*_1$ be the ``H''-graph with vertices $u, v, x_1, x_2, y_1, y_2$ and edges $uv$, $ux_1$, $ux_2$, $vy_1$, $vy_2$. From this graph we obtain $H_i^*$ by subdividing the edge $uv$ $i-1$ times.

The \emph{distance} $\dist(u,v)$ of two vertices $u$ and $v$ is the length of a shortest path between them. The \emph{eccentricity} of a vertex $v$ is the maximum distance between~$v$ and any other vertex of $G$. The \emph{radius} of $G$ is the minimum eccentricity of any vertex in $G$ and the \emph{diameter} the maximum eccentricity of any vertex in $G$.
We need the following result from the literature.
\begin{theorem}[\cite{CS16}]\label{T-Pk}
    Let $G$ be a connected $P_k$-free graph, for $k\geq 4$. Then $G$ either has a dominating $P_{k-2}$ or a dominating $P_{k-2}$-free connected subgraph and such a subgraph can be found in polynomial time.
\end{theorem}

\noindent
A \emph{red-blue colouring} of $G$ is a colouring of the vertices with red and blue, using each colour at least once. We say that a red-blue colouring is \emph{valid} if every red vertex is adjacent to at most~$1$ blue vertex and every blue vertex is adjacent to at most~$1$ red vertex.
An edge is \emph{bichromatic} if it has one red and one blue endvertex, and an edge or a vertex set $S \subseteq V$ is \emph{monochromatic} if all vertices in~$S$ have the same colour. Let $R$ be the set of red vertices of $G$. A \emph{red component} of $G$ is a connected component of the graph $G[R]$. Similarly, for $B$, the set of blue vertices, a \emph{blue component} of $G$ is a connected component of $G[B]$.
The \emph{value}~$\nu$ of a valid red-blue colouring is the number of bichromatic edges.
We make several well-known observations, see e.g.~\cite{LPR22,LPR24}.

\begin{observation}\label{basic-observations}
    Let $G$ be a graph, $\ell \geq 3$, and $r\geq 2$. Then,
    \begin{itemize}
        \item $G$ has a matching cut of size~$\nu$ if and only if it has a red-blue colouring of value~$\nu$.
        \item $K_\ell$ and $K_{\ell,r}$ are monochromatic in any valid red-blue colouring.
        \item if $G$ has a vertex $v$ of degree~$1$, the partition of the vertices into the sets $\{v\}$ and $V\setminus \{v\}$ leads to a matching cut of size~$1$.
    \end{itemize}
\end{observation}

\noindent
Let $R,B \subseteq V$ with $ R \cap B = \varnothing$. A \emph{red-blue $(R,B)$-colouring} is a partial colouring of $G$ where the vertices in $R$ are coloured red and those in $B$ are coloured blue.
We say that a red-blue $(R,B)$-colouring is \emph{valid} if it can be extended to a valid red-blue colouring.
Given a pair $(R,B)$, we define propagation rules that try to extend the red-blue $(R,B)$-colouring or show that no valid red-blue $(R,B)$-colouring exists. Some of these rules have already been given in~\cite{LL19} and have been widely used since then. In particular R4, which does not hold when maximising the cut, is crucial for our results.
Let $R' \subseteq R$ be the set of red vertices which have a blue neighbour in $B$ and similarly, let $B' \subseteq B$ be the set of blue vertices with a red neighbour in $R$. Let $Z = V\setminus (R\cup B)$.

\begin{description}
    \item[R1] Return \textsf{no}, i.e., $G$ has no valid red-blue $(R,B)$-colouring, if a vertex $v \in Z$~is:
    \begin{enumerate}
        \item adjacent to a vertex in $R'$ and a vertex in $B'$.
        \item adjacent to two vertices in $R$ and a vertex in $B'$.
        \item adjacent to two vertices in $B$ and a vertex in $R'$.
        \item adjacent to two vertices in $R$ and two vertices in $B$.
    \end{enumerate}

    \item[R2] Let $v\in Z$:
    \begin{enumerate}
        \item if $v$ is adjacent to two vertices in $R$ or one in $R'$, then colour $v$ red.
        \item if $v$ is adjacent to two vertices in $B$ or one in $B'$, then colour $v$ blue.
    \end{enumerate}
    \item[R3] Let $H$ be a subgraph of $G$, isomorphic to $K_{3}$, let $v \in Z \cap V(H)$.\\
        If $H$ contains a blue (resp.~red) vertex, then colour $v$ blue (resp.~red).

    \item[R4] Let $C$ be a connected component of $G[Z]$:
        \begin{enumerate}
            \item if no vertex of $C$ is adjacent to a vertex in $R$, then colour $C$ blue.
            \item if no vertex of $C$ is adjacent to a vertex in $B$, then colour $C$ red.
        \end{enumerate}

    \item[R5] Let $H$ be a subgraph of $G$, isomorphic to $K_{2,3}$, let $v \in Z \cap V(H)$.
        \begin{enumerate}
            \item If $H$ contains a blue and a red vertex, return \textsf{no}.
            \item If $H$ contains a blue (resp.~red) vertex, then colour $v$ blue (resp.~red).
        \end{enumerate}

\end{description}

\noindent
We say that a propagation rule is \emph{safe} if $G$ has a valid red-blue $(R,B)$-colouring before the application of the rule if and only if $G$ has a valid red-blue $(R,B)$-colouring after the application of the rule. 

\begin{restatable}{lemma}{lempropsafe}\label{l-propsafe}\label{l-propconnectivity}
    Propagation rules R1--R5 are safe, can be applied in polynomial time, their application creates no additional red and blue components and does not increase the minimum possible value of a valid red-blue $(R,B)$-colouring.
\end{restatable}

 \begin{proof}
    The safeness for R1 and R2 has already been proven in~\cite{LL19}.
    Since by Observation~\ref{basic-observations} triangles and $K_{2,3}$ are monochromatic, it follows that R3 and R5 are safe.
    To see that R4 is safe, recall that we start with at least one vertex of each colour. This guarantees the existence of vertices of both colours and, therefore, we may colour connected components adjacent to vertices of only one colour in that colour. 
    For every application of a propagation rule, we only colour vertices with the colour of one of their neighbours. Therefore, we do not create any additional coloured components.
    Note further that R1--R3 and R5 preserve the value of the colouring. R4 only decreases the maximum possible value but never results in an increase of the minimum possible value.
\end{proof}

\noindent
The following lemmas have been shown previously for \maxmc{} in~\cite{LPR24}.

\begin{restatable}{lemma}{lemsmalldomset}\label{l-smalldomset}
    Let $G$ be a connected graph with domination number~$g$, where $g\geq 1$ is a constant. We can find a minimum red-blue colouring of $G$ or conclude that no such colouring exists in time $O(2^g n^{g+2})$.
\end{restatable}

\begin{proof}
Let $D$ be a dominating set of $G$ with $|D| = g$. We branch over all $2^{|D|} = 2^g$ options of
colouring the vertices of $D$ red or blue. 
For every coloured vertex without a neighbour of the other colour, we consider all $O(n)$ options to colour its neighbourhood. For every vertex~$v$ with a neighbour of the other colour, we colour all its other neighbours with the colour of~$v$.
Since $D$ is a dominating set, we guessed a red-blue colouring of the whole graph $G$. We can check in time $O(n^2)$ whether this colouring is a valid red-blue colouring and compute its value. If the colouring is not valid, we discard the branch, otherwise we remember its value and output the valid colouring with the smallest value. Since we consider $O(2^gn^g)$ colourings, we get a runtime of $O(2^g n^{g+2})$. 
\end{proof}

\begin{restatable}{lemma}{lemindepset}\label{L-Indep Set}
    Let $G = (V,E)$ be a connected graph with a red-blue $(R,B)$-colouring, for $R,B \subseteq V$. If $V\setminus (B\cup R)$ induces an independent set then it is possible in polynomial time to either find a minimum valid red-blue $(R,B)$-colouring or to conclude that $G$ has no valid red-blue $(R,B)$-colouring.
\end{restatable}

\begin{proof}
    Let $Z = V\setminus (B\cup R)$. According to R4, we colour every vertex of $Z$ that is only adjacent to vertices of one colour with this colour. Let $U$ be the set of the remaining uncoloured vertices of $Z$, i.e., the vertices with neighbours in both $R$ and $B$. Let $W = N(U)$. 
    Let $G'$ be the graph with vertex set $U\cup W$ and edge set $E(G') = \{uv | u \in U, v \in W \}$. We claim that the set of bichromatic edges of every valid red-blue $(R,B)$-colouring of $G$ is the union of a perfect matching in~$G'$ and the set of edges with one endvertex in $R$ and one in $B$.

\begin{figure}
    \centering
    \begin{tikzpicture}
    \tikzset{
        rahmen1/.style={
            rounded corners = 5pt,
            draw,
            dashed,
            minimum width=20pt,
            minimum height=45pt
        }
    }
    \tikzset{
        rahmen2/.style={
            rounded corners = 5pt,
            draw,
            dashed,
            minimum width = 15 pt,
            minimum height = 55 pt
        }
    }
    \tikzset{
        rahmen3/.style={
            rounded corners = 5pt,
            draw,
            dashed,
            minimum width = 20pt,
            minimum height = 116pt
        }
    }
    \tikzset{
        rahmen4/.style={
            rounded corners = 5pt,
            draw,
            dashed,
            minimum width = 20pt,
            minimum height = 60pt
        }
    }
    \begin{scope}[yscale = 0.5, xscale = 0.7]
        \node[rvertex](r3) at (-2,3){};
        \node[rvertex](r4) at (-2,2){};
        \node[rvertex](r5) at (-2,1){};
        
        \node[bvertex](b1) at (2,3){};
        \node[bvertex](b2) at (2,2){};
        \node[bvertex](b4) at (2,0){};
        
        \node[evertex](v4) at (0,3) {};
        \node[evertex](v5) at (0,2) {};
        \node[evertex](v6) at (0,1) {};
        \node[evertex](v7) at (0,0) {};
        
        \draw[edge] (r3) -- (v4);
        \draw[tedge] (r3) -- (v5);
        \draw[edge] (r4) -- (v6);
        \draw[tedge] (r5) -- (v7);
        
        \draw[tedge] (b1) -- (v4);
        \draw[edge] (b2) -- (v5);
        \draw[tedge] (b2) -- (v6);
        \draw[edge] (b4) -- (v7);

        \node[rahmen1, nicered] (rm1) at (-2,2){};
        \node[rahmen4] (rm2) at (0,1.5){};
        \node[rahmen4, lightblue] (rm4) at (2,1.5){};
        \node[](x) at (-2,4){$X$};
        \node[](y) at (2,4){$Y$};
        \node[](u) at (0,4){$U$};
        \end{scope}

    \begin{scope}[yscale = 0.5, xscale = 0.7, shift = {(7,0)}]
        \node[rvertex](r3) at (-2,3){};
        \node[rvertex](r4) at (-2,2){};
        \node[rvertex](r5) at (-2,1){};
        
        \node[bvertex](b1) at (2,3){};
        \node[bvertex](b2) at (2,2){};
        \node[bvertex](b4) at (2,0){};
        
        \node[rvertex](v4) at (0,3) {};
        \node[bvertex](v5) at (0,2) {};
        \node[rvertex](v6) at (0,1) {};
        \node[bvertex](v7) at (0,0) {};
        
        \draw[edge] (r3) -- (v4);
        \draw[edge] (r3) -- (v5);
        \draw[edge] (r4) -- (v6);
        \draw[edge] (r5) -- (v7);
        
        \draw[edge] (b1) -- (v4);
        \draw[edge] (b2) -- (v5);
        \draw[edge] (b2) -- (v6);
        \draw[edge] (b4) -- (v7);

        \node[rahmen1, nicered] (rm1) at (-2,2){};
        \node[rahmen4] (rm2) at (0,1.5){};
        \node[rahmen4, lightblue] (rm4) at (2,1.5){};
        \node[](x) at (-2,4){$X$};
        \node[](y) at (2,4){$Y$};
        \node[](u) at (0,4){$U$};
        \end{scope}
    \end{tikzpicture}
    \caption{A perfect matching in $G'$ (left) and the corresponding minimum valid red-blue colouring.}\label{fig-lemma Indep Set}
\end{figure}
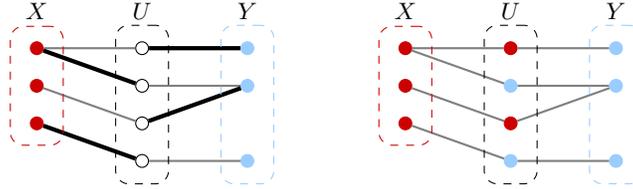

    First suppose that $G'$ has a perfect matching $M$. Let $z\in U$. Suppose that $z$ is incident to an edge $zw\in M$. If $w\in R$ we colour $z$ in  blue. Otherwise, $w\in B$ and we colour $z$ red. Since $M$ is perfect, we coloured every vertex in $U$ and every vertex in $W$ got at most one neighbour of the other colour. We obtain a valid red-blue colouring of $G$.

    Now suppose that $G$ has a valid red-blue $(R,B)$-colouring. Every edge with an endvertex in $R$ and the other one in $B$ is bichromatic and there are no other bichromatic edges in $G[R\cup B]$. Let $M$ be the set of bichromatic edges in the remaining graph. Note that every edge of $M$ has an endvertex in $U$ so $M$ is in~$G'$. By definition of a red-blue colouring, $M$ has to be a matching. Moreover, if $z\in U$ then $z$ has a blue neighbour and a red neighbour, so it is contained in a bichromatic edge. We conclude that $M$ is perfect.
    
    Since $G'$ is bipartite we can determine a maximum matching of $G'$ in polynomial time using 
    Karp's algorithm~\cite{HK73}. Note that every perfect matching in~$G'$ has the same size. It follows that we can find a minimum valid red-blue $(R,B)$-colouring or conclude that there is no valid red-blue $(R,B)$-colouring in polynomial time.
\end{proof}

\noindent
We show more lemmas specifically for red-blue colourings in $P_8$-free graphs.
To show the following lemma, we will make use of a result from the literature.
\begin{lemma}[\cite{LPR22}]\label{l-LPR22-monodom}
    Let $G$ be a connected graph with a red-blue $(R,B)$-colouring, for connected sets $R, B \subseteq V$. If $R$ or $B$ dominates the uncoloured vertices, it is possible to check in polynomial time if $G$ has a valid red-blue $(R,B)$-colouring.
\end{lemma}

\begin{restatable}{lemma}{lemmonodom} \label{L-monodom}
    Let $G = (V,E)$ be a connected $P_8$-free graph with a red-blue $(R,B)$-colouring, $R,B \subseteq V$, $R, B \neq \varnothing$, and $G[R]$ and $G[B]$ connected. Let $Z$ be the set of uncoloured vertices. If $Z$ is dominated by $R$ or $B$, then we can find in polynomial time a minimum red-blue $(R,B)$-colouring of~$G$ or decide that no such colouring exists. 
\end{restatable}

\begin{proof}
Let $R$ (resp.~$B$) be the connected set of red (resp.~blue) vertices of $G$.
Let $Z$ be the set of uncoloured vertices. 
We may assume without loss of generality that $R$ dominates~$Z$.
Hence, every connected component of $G[Z]$ is monochromatic in every valid red-blue $(R,B)$-colouring of $G$. To see this, note that a blue vertex with a red neighbour in $G[Z]$ would have a second red neighbour in $R$, a contradiction.

Suppose every vertex in $Z$ is adjacent to a blue and a red vertex. This implies that in every valid red-blue $(R,B)$-colouring of $G$, every vertex in $Z$ contributes~$1$ to the value of the colouring. Therefore, the value of all valid red-blue $(R,B)$-colourings is the same, and we apply Lemma~\ref{l-LPR22-monodom}. We either find a valid red-blue colouring or conclude that no such colouring exists. In the first case we remember its value, in the latter, we discard the branch.

Hence, in the following we may assume that there is a vertex $x \in Z$ which is not adjacent to a vertex in $B$. Let $C$ be the connected component of $x$ in~$G[Z]$. 
Since $C$ was not coloured by R4, there is at least one vertex $y$ in $C$ adjacent to a blue vertex $b$ in $B$. 
If there are several such vertices, we take $y$ to be the closest to $x$ in $C$. Let $x = p_0 \, \dots \, p_k = y$ be the shortest path between $x$ and $y$ in $C$. Since $G$ is $P_8$-free, we have that $k \leq 6$. Let $q_0,\dots,q_k$ be the red neighbours of $p_0,\dots,p_k$, in the dominating set $R$, see Figure~\ref{f-p7-dom-poly}. We branch over the $\mathcal{O}(n^7)$ colourings of $N(q_0)\cup\dots\cup N(q_k)$ and propagate the colouring. We leave $C$ uncoloured even though we already know its colour.

    \clapp{\ref{L-monodom}}{1}
        All connected components of $G[Z]$ containing a vertex without a blue neighbour are adjacent to $b$.

\begin{claimproof} 
    Suppose for a contradiction that there is a connected component $C'$ of $G[Z]$ containing a vertex without a blue neighbour and such that $C'$ is not adjacent to $b$. Note that R4 implies that there is a blue neighbour $b' \in B$ of $C'$ and thus, $C'$ consists of at least $2$ vertices.
    If $C'$ has several blue neighbours, we let $b'$ be the one with the smallest distance in $B$ to~$b$. Let $b = b_0 \, \dots \, b_\ell=b'$, for $\ell\leq 6$, be the shortest blue path between $b$ and $b'$. 
Further, let $u \in C'$ be a neighbour of $b'$, let $v \in C'$ be a neighbour of $u$ and let $q_u$ and $q_v \in R$ be the red neighbours of $u$ and $v$, see Figure~\ref{f-p7-dom-poly}.     

\begin{figure}
    \centering
    \begin{tikzpicture}
\begin{scope}
    \node[evertex, label={below:$p_k = y$}](s2p) at (2,0){};
    \node[evertex, label=below:$p_1$](s1p) at (1,0){};
    \node[evertex, label={below:$x = p_0$}](s0p) at (0,0){};

    \node[rvertex, label = right:$q_k$](q2p) at (2,1){};
    \node[rvertex, label = right:$q_1$](q1p) at (1,1){};
    \node[rvertex, label = right:$q_0$](q0p) at (0,1){};

    \draw[edge](s0p) -- (s1p);
    \draw[edge](s0p) -- (q0p);
    \draw[edge, dotted](s1p) -- (s2p);
    \draw[edge](s1p) -- (q1p);
    \draw[edge](s2p) -- (q2p);

    \draw[] (-0.6, -0.6) rectangle (2.6,0.35);
    \node[] (c) at (-0.85,0){$C$};
    
\end{scope}

\begin{scope}[shift = {(3.5,0)}]
    \node[bvertex, label = {below:$b = b_0$}](t0) at (0,0){};
    \node[bvertex, label = below:$b_1$](t1) at (1,0){};
    \node[bvertex, label = {below:$b_\ell = b' $}](tk) at (2,0){};

    \draw[edge](t0) -- (t1);
    \draw[edge, dotted](t1) -- (tk);

    \draw[lightblue] (-0.6,-0.6) rectangle (2.6,0.35);

\end{scope}

\begin{scope}[shift = {(8,0)},  xscale = -1]
    \node[evertex, label=below:$v$](s2) at (0,0){};
    \node[evertex, label=below:$u$](s1) at (1,0){};

    \node[rvertex, label = left:$q_v$](q2) at (0,1){};
    \node[rvertex, label = left:$q_u$](q1) at (1,1){};

    \draw[edge](s1) -- (s2);
    \draw[edge](s1) -- (q1);
    \draw[edge](s2) -- (q2);

    \draw[] (-0.5, -0.6) rectangle (1.5,0.35);
    \node[] (c) at (-0.75,0){$C'$};
    
\end{scope}

    \draw[edge](s2p) -- (t0);
    \draw[edge](s1) -- (tk);

\end{tikzpicture}
    \caption{The connected components $C$ and $C'$ of $G[Z]$.}
    \label{f-p7-dom-poly}
\end{figure}
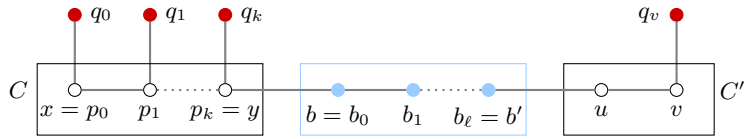

    We claim that $p_0 \, \dots \, p_k \, b_0 \, \dots \, b_\ell\, u\, v \, q_v$ is an induced path of length at least~$7$. Note first that, by assumption, $x = p_0$ is not adjacent to any of $b_0,\dots,b_\ell$. 
   Further, the choice of $y = p_k$ implies that none of $p_1, \dots, p_{k-1}$ is adjacent to any of $b_0,\dots,b_\ell$.
   Any uncoloured neighbour of $y$ is in $C$ and remains uncoloured. Therefore, and since $y$ did not have two blue neighbours before colouring the neighbourhood of $q_0, \dots, q_k$, $y$ has no neighbour in $b_1, \dots, b_\ell$.
    Note further that no vertex of $C$ can be adjacent to a vertex in $C'$.
   If $q_v$ is a neighbour of one of $p_0, \dots , p_k$ and was uncoloured before colouring the neighbours of $q_0, \dots, q_k$, it would belong to $C$ and would thus remain uncoloured. 
   Also, if $q_v$ would be one of the red neighbours of $p_0, \dots , p_k$ that is, $q_v$ is one of $q_0, \dots, q_k$, $v$ would have been coloured. For the same reason, $q_v$ cannot be adjacent to any of the $b_0, \dots, b_\ell$ or to $u$.
Therefore, $p_0 \, \dots \, p_k \, b_0 \, \dots \, b_\ell\, u\, v \, q_v$ is indeed an induced path of length at least~$7$.
It also follows that $k = \ell = 1$.

We now branch over all $O(n^4)$ colourings of $N(q_u), N(q_v), N(b), N(b')$. We propagate the colouring and leave $C$ and $C'$ uncoloured even though we already know their colour.
Let $Z$ be again the set of uncoloured vertices. If there is no connected component of $G[Z]$ other than $C$ and $C'$ we colour them. If this leads to a valid red-blue colouring, we remember its value, else we discard the branch. In both cases we consider the next branch.
Hence, we assume that there is another connected component $C''$ in $G[Z]$.
Let $w$ be a vertex in such a connected component. Note that $w$ has a neighbour $q_w$ in $R$. We consider the shortest red path from $w$ to $q_v$. If it contains no vertex of $q_0,  q_1$, we get that $w\, q_w \, \dots \, q_v\, v\, u \, b'  \, b \, y \, x$ is an induced path of length at least~$8$, a contradiction to the assumption that $G$ is $P_8$-free. 
Otherwise we obtain that the path $q_w \, \dots \,q_v $ contains at least one extra vertex and thus, $w\, q_w \, \dots \, q_v\, v\, u \, b' \, b $ is an induced path of length at least~$8$, again a contradiction.

It follows that $C''$ does not exist. If $C$ and $C'$ were the only components of $G[Z]$, they got coloured.
Hence, we obtain that $C'$ does not exist.
Thus, all connected components of $G[Z]$ containing a vertex without a blue neighbour are adjacent to $b$ and the claim follows.
\end{claimproof}

    \noindent
    By Claim~\ref{L-monodom}.1, we may assume that if there is any uncoloured vertex without a blue neighbour, then there is a vertex $b \in B$ which is adjacent to all connected components containing such a vertex.
    Recall that the connected components of $G[Z]$ are monochromatic. Further, $b$ is adjacent to at most one red vertex. Hence, we branch over the $O(n)$ colourings of the components adjacent to $b$ and propagate the colouring using R1--R5. If we obtain a no-answer, we discard the branch.
    Afterwards, every uncoloured vertex has a red and a blue neighbour.
    Hence, every vertex in $Z$ contributes~$1$ to the value of every valid red-blue colouring. As before, we apply Lemma~\ref{l-LPR22-monodom}. If we get a valid red-blue $(R,B)$-colouring, we remember its value, otherwise, we discard the branch.

Whenever we discard a branch, we consider the next. 
We remember the value of every valid red-blue colouring which we obtain and return the minimum of their values or that no such colouring exists.
The correctness of our algorithm follows from its description.
We consider in total $O(n^{12})$ branches, each of which can be processed in polynomial time. Therefore, our algorithm runs in polynomial time.
\end{proof}

\begin{restatable}{lemma}{lemddmp}\label{l-ddmp}
    Let $G$ be a connected $P_8$-free graph with a connected dominating set~$D$. Let $D'$ be a dominating set of $G[D]$ of size $g$. We can find in polynomial time a minimum red-blue colouring of $G$ in which $D'$ is monochromatic or conclude that no such colouring exists.
\end{restatable}

 \begin{proof}
    Let $D$ be a connected dominating set of $G$ and $D'$ a dominating set of $G[D]$ of size~$g$.
    Without loss of generality, we colour $D'$ red.
    We branch over all $\mathcal{O}(n^{\strut g} )$ colourings of the neighbourhood of $D'$. If $D'$ is monochromatic we branch over all $O(n)$ options to colour some vertex blue. Note that this implies that $D$ is coloured.
    We propagate the colouring exhaustively using R1--R5. If we obtain a no-answer we discard the branch, otherwise we continue.
    Let $R$ be the set of red vertices and $B$ the set of blue vertices.

    Since $D$ is dominating and coloured, every uncoloured vertex has a neighbour in $D$.
    Note that every blue vertex in $D$ has a red neighbour in $D'$. Thus, the neighbours of the blue vertices in $D$ are coloured blue by propagation. Therefore, every uncoloured vertex has exactly one red neighbour in $D$, resulting in a monochromatic dominating set. 
    We apply Lemma~\ref{L-monodom} to either find a minimum red-blue $(R,B)$-colouring whose value we remember or conclude that no such colouring exists and discard the branch. 

    Whenever we discard a branch, we consider the next. If no branch returns a valid red-blue colouring we return that no such colouring exists, otherwise we return the obtained red-blue colouring with the minimum value.

    The correctness of our algorithm follows from its description. Every branch can be processed in polynomial time, and we consider $O(n^{\strut g})$ branches. Therefore, our algorithm runs in polynomial time.
\end{proof}

\section{Polynomial-Time Results}
\label{sec:poly}
Before showing our main results in this section, we first state several results where the proofs for \maxmc{} can be applied to \minmc.
The polynomial-time algorithms for \maxmc{} for $P_6$-free graphs, graphs of diameter~$2$~(\cite{LPR24}) and bipartite graphs of radius~$2$~(\cite{Lu25}) all start with branching on all colourings of a dominating set of $G$. Then, structural observations are made that show that the uncoloured vertices form an independent set. None of these steps uses that the matching cut is maximised. Hence, we apply the same arguments and conclude with Lemma~\ref{L-Indep Set}.
This approach can also be applied to show Theorem~\ref{T-H+P2}, based on another proof from~\cite{LPR24}.
The polynomial-time algorithm in~\cite{Lu25} for bipartite graphs of diameter~$3$ finds all matching cuts. Hence, we can easily use it to find a minimum matching cut. We thus obtain the following.

\begin{theorem}\label{T-P6}
    \minmc{} is solvable in polynomial time 
    \begin{itemize}
        \item for $P_6$-free graphs,
        \item for graphs of diameter at most $2$,
        \item for bipartite graphs of radius at most~$2$, and
        \item for bipartite graphs of diameter at most~$3$.
    \end{itemize}
\end{theorem}

\begin{theorem}\label{T-H+P2}
    Let $H$ be a graph. If \minmc{} is polynomial-time solvable for $H$-free graphs, then it is so for $(H+P_2)$-free graphs.
\end{theorem}

\noindent
As can be seen from Table~\ref{tab:overview}, for \maxmc{} these are the maximal polynomial cases for $H$-free graphs. This stands in contrast to our main results for \minmc. In the following we show the polynomial-time solvability for $S_{1,1,3}$-free, $(P_6+P_4)$-free and $P_8$-free graphs. Note that all three results solve open cases for \mc.

\begin{restatable}{theorem}{thmsooe}\label{t-s113}
    \minmc{} is solvable in polynomial time for $S_{1,1,3}$-free graphs.
\end{restatable}

\begin{proof}
    Let $G = (V,E)$ be a $S_{1,1,3}$-free graph.
    We first check whether $G$ has a vertex of degree~$1$. If this is the case, then, by Observation~\ref{basic-observations}, we find a minimum matching cut. 
    Therefore, we may assume that $G$ has minimum degree $2$. 
    We apply Observation~\ref{basic-observations} and search for a minimum red-blue colouring of $G$.

    We branch over all $O(n^2)$ options of choosing two adjacent vertices $x$ and $y$ of $G$. We colour $x$ red and $y$ blue. We propagate the red-blue $(\{x\}, \{y\})$-colouring exhaustively which takes polynomial time by Lemma~\ref{l-propsafe}. If we get a no-answer, we discard the branch, otherwise, we obtained a red set $R$ and a blue set $B$.
    Let $Z$ be the set of uncoloured vertices. Note that since $G$ has minimum degree~$2$, $y$ has at least one neighbour other than~$x$, say $y'$, which is coloured blue by propagation.
    In addition the path $x\,y\,y'$ is induced since otherwise $x$ has two blue neighbours, a contradiction.
    We first show the following claim.

\clapp{\ref{t-s113}}{1} 
        Either $G[Z]$ is claw-free or $N(y) = \{x,y'\}$ and every claw in $G[Z]$ has a vertex adjacent to $y'$.
    
    \begin{claimproof}
    Suppose there is a claw $C$ in $G[Z]$, consisting of vertices $u, v_1, v_2, v_3$ such that $u$ is adjacent to $v_1, v_2, v_3$.

    Consider a shortest path without blue vertices from $C$ to $x$. 
    Note that such a path always exists, otherwise $C$ would be blue by R4.
    Suppose that there is such a path starting from~$u$. Let $u = p_0 \, p_1 \, \dots \, p_k = x$ be this path and set $p_{k+1} = y$. Since all neighbours of $y$ except $x$ are blue, the path $p_0 \, \dots \, p_{k+1}$ is induced. If $u$ was adjacent to $x$ it would be coloured by~R2. Since it is uncoloured we thus have $k \geq 2$.
    Note that $p_2, p_3, \dots$ are not adjacent to any vertex in $C$, in particular, $p_1\not\in\{v_1,v_2,v_3\}$, otherwise there would be a shorter path from $C$ to $x$.
    
    If at least two of $v_1, v_2, v_3$, say $v_1$ and $v_3$, are not adjacent to $p_1$, then $u, v_1, v_3, p_1, p_2, p_3$ induce $S_{1,1,3}$. Hence, we may assume that $v_1$ and $v_3$ are adjacent to $p_1$. 
    Hence, $p_1, u, v_1$ and $p_1, u, v_3$ are triangles.
    If $p_1$ is adjacent to $x$ it is coloured red and thus, by R3, $u,v_1$ and $v_3$ are coloured, a contradiction to the assumption that $V(C) \subseteq Z$.
    Thus, we have $k\geq 3$ and that $p_1, v_1, v_3, p_2, p_3, p_4$ induce $S_{1,1,3}$, a contradiction.

    \medskip
    \noindent
    It remains to consider the case where there is no shortest path without blue vertices from~$C$ to $x$ that starts from $u$. We may assume without loss of generality that there is such a path starting from $v_1$. Let $v_1 = p_0 \, \dots \, p_k = x$ be this path and we again set $p_{k+1} = y$. As before, this path is induced and $p_2,p_3,\dots$ are not adjacent to any vertex in $C$. Since $v_1$ is uncoloured, we further get $k \geq 2$. 
    
    If $p_1$ is neither adjacent to $v_2$ nor to $v_3$ then $u,v_1,p_1,p_2, v_2,v_3$ is an induced $S_{1,1,3}$. 
    Thus, $p_1$ is adjacent to at least one of $v_2$ and $v_3$. We first consider the case where $p_1$ is adjacent to both $v_2$ and $v_3$. 
    If $k = 2$, then $p_1$ is a neighbour of $x$ and thus coloured red. Note that $p_1,u, v_1, v_2, v_3$ induce $K_{2,3}$. It follows from R5 that $C$ is coloured, a contradiction.
    Otherwise, $k \geq 3$. In this case $p_1,p_2,p_3,p_4, v_1, v_3$ induce $S_{1,1,3}$, a contradiction.
    
    We are left with the case where $p_1$ is adjacent to exactly one of $v_2$ and $v_3$. 
    Without loss of generality we may assume $p_1$ is adjacent to $v_3$ and not adjacent to $v_2$.
    Again, if $k \geq 3$, we find $S_{1,1,3}$ consisting of $v_1, v_3, p_1,\dots, p_4$. Else, we may assume that $k = 2$.
    We now consider the neighbour $y'$ of $y$. Note first that $y'$ is not adjacent to $x$, since $y$ is the only blue neighbour of~$x$.
    Further, if $p_1$ and $y'$ are adjacent, both have a neighbour of the other colour. Hence, their neighbourhoods, including $v_1,v_3$ are coloured by propagation, a contradiction to $C$ being uncoloured. 
    If $y'$ is adjacent neither to $v_1$ nor to $v_3$, we find an induced $S_{1,1,3}$ consisting of $p_1,v_1,v_3,x,y,y'$, a contradiction.
    Hence, $y'$ is adjacent to at least one of $v_1,v_3$. Without loss of generality we may assume $y'$ is adjacent to $v_3$.
    If $y'$ is adjacent to $u$, then $y', u, v_3$ form a triangle. Since $y'$ is coloured blue, by R3, $u$ and $v_3$ are coloured blue as well, a contradiction. Hence, $y'$ is not adjacent to $u$.

    If $y'$ is adjacent to none of $v_1, v_2$, we obtain $S_{1,1,3}$ induced by $u,v_1,v_2,v_3,y', y$. Hence, $y'$ is adjacent to at least one of them.
    Suppose for a contradiction that $y' v_1 \in E$. Recall that $p_1,x$ are red and $y, y'$ are blue.
    Since $v_1$ is adjacent to $p_1$ and $y'$, it has a red and a blue neighbour, so its third neighbour $u$ must have the same colour as $v_1$. By symmetry, the same holds for $v_3$, so $v_1, u, v_3$ are all coloured the same. If they are all red, the blue vertex $y'$ has two red neighbours, if they are all blue, the red vertex $p_1$ has two blue neighbours. In both cases we get a contradiction and thus, $v_1$ is not adjacent to $y'$. Note that this implies that $v_2$ is adjacent to $y'$.

    \begin{figure}
        \centering
        \begin{tikzpicture}
    \node[rvertex, label = below: ${x = p_2}$](x) at (0,0){};
    \node[rvertex, label = below: $p_1$](p1) at (1,0){};

    \node[bvertex, label = above: ${y = p_3}$](y) at (0,1){};
    \node[bvertex, label = above: $y'$](yp) at (1,1){};

    \node[evertex, label = above:${v_1}$](v1) at (2,0.25){};
    \node[evertex, label = above:$u$](u) at (3,0.25){};
    \node[evertex, label = right:$v_3$](v3) at (4,0.25){};
    \node[evertex, label = $v_2$](v2) at (2.5,0.75){};

    \draw[edge](x) -- (y);
    \draw[edge](y) -- (yp);
    \draw[edge](x) -- (p1);
    \draw[edge](p1) -- (v1);
    \draw[edge](v1) -- (u);
    \draw[edge](v2) -- (u);
    \draw[edge](v3) -- (u);
    \draw[edge](p1) to [bend right = 10] (v3);
    \draw[edge](v2) -- (yp);
    \draw[edge](v3) to [bend right = 20] (yp);
\end{tikzpicture}
        \caption{An illustration of the structure occurring in the proof of Theorem~\ref{t-s113}.}
        \label{f-s113}
    \end{figure}

    We are left with the situation as depicted in Figure~\ref{f-s113}.
    Suppose that $y$ has a second blue neighbour $y''$.
    Applying the arguments given for $y'$, we conclude that $y''$ is adjacent to $v_1,v_2$ or to $v_2,v_3$.
    Note that in both cases, $v_2$ has two blue neighbours, $y'$ and $y''$ and is thus coloured blue, a contradiction to $C$ being uncoloured.
    Hence, $y''$ does not exist. Thus, $y$ has degree exactly $2$ and $v_2, v_3$ are adjacent to $y'$.
    This proves the claim. 
    \end{claimproof}

\noindent
By Claim~\ref{t-s113}.1 we get that either $G[Z]$ is claw-free or $y$ has degree exactly $2$ and every claw in $G[Z]$ is adjacent to the blue neighbour $y'$ of $y$. In the latter case, we branch over all $O(n)$ options to colour the neighbourhood of $y'$ and propagate the resulting colouring. Since every claw in $G[Z]$ is adjacent to $y'$ this will colour at least one vertex of each claw.
We update $Z$ to be again the set of uncoloured vertices and conclude that now $G[Z]$ is claw-free.

We need another short claim.

\clapp{\ref{t-s113}}{2}
    No component of $G[Z]$ is an induced cycle on at least $4$ vertices.
\begin{claimproof}
    Suppose for a contradiction that $G[Z]$ has a component $C$ which is a cycle with at least $4$ vertices. Every vertex in $C$ has at most one red and one blue neighbour which is neither $x$ nor $y$. Two neighbours on the cycle do not have a common coloured neighbour, since otherwise they would have been coloured.
     We consider a shortest path from $C$ to $x$ avoiding blue vertices. Let $v$ be the vertex on $C$ where the path starts.
     Note that this path contains at least one vertex between $x$ and $v$, since the neighbourhood of $x$ is coloured, but $v$ is not. Further, we can add $y$ to the path and still obtain an induced path since all neighbours of $y$ except $x$ are coloured blue.
     Thus, we find an induced $S_{1,1,3}$ by taking $v$, its two neighbours on the cycle, and three vertices on the path from $v$ to $y$ via $x$, a contradiction.
     Hence, the claim follows.
\end{claimproof}

    \noindent
    Let now $Z_\Delta$ be the set of edges of $G[Z]$ which are contained in some triangle. 
    Let $C_1, \dots, C_r$ be the connected components of $G[Z_\Delta]$.
    Then, each of $C_1, \dots, C_r$ is monochromatic.
    We apply propagation rule R2 on each of the sets $C_1, \dots, C_r$ to potentially extend them with other vertices which will necessarily get the same colour. If due to the propagation two sets $C_1, \dots, C_r$ share a vertex, we unite them and continue the propagation with this set.
    This propagation might change the number of sets. After the propagation, we update the integer~$r$ and assume that we have $r$ sets $C_1, \dots, C_r$.
    
    Let $Z'\subseteq Z$ be the set of edges not contained in $Z_\Delta$.
    Since $G[Z]$ is claw-free and not a cycle, the graph $G[Z']$ has maximum degree~$2$ and is a collection of paths.
    We contract each of the sets $X,Y, C_1, \dots, C_r$ to a single vertex and obtain a graph $G'$. Note that this might result in a multigraph.

    \clapp{\ref{t-s113}}{3}
        $G$ has a red-blue $(x,y)$-colouring of value~$\mu$ if and only if $G'$ has an $(x,y)$-cut of size~$\mu$.
    \begin{claimproof}
        Suppose $G$ has a red-blue $(x,y)$-colouring of value~$\mu$. Recall that every component $X,Y, C_1, \dots, C_r$ is monochromatic, that is, no bichromatic edge is inside one of these components. Thus, every bichromatic edge in $G$ corresponds to an edge in~$G'$, and the cut in $G$ induced by the bichromatic edges is a cut in~$G'$ of the same size.

        For the other direction suppose that $G'$ has an $(x,y)$-cut of size~$\mu$. This cut partitions the vertex set into two sets. We colour the set containing $x$ red and the set containing $y$ blue. Recall that the graph $G[Z']$ is a collection of disjoint paths. Thus, every vertex in $G$ is adjacent to at most one vertex of the other colour and the colouring is valid. Note that the value of the colouring equals the size of the cut in $G'$.
    \end{claimproof}
    
    \noindent
    We follow this approach for every pair of vertices $x,y$. Whenever we get a contradiction, we discard the branch. For every branch that leads to a valid red-blue colouring, we remember the value of this colouring and minimize over all values of all valid red-blue colourings obtained in the process.

    The correctness of our algorithm follows from its description. Regarding the running time, note that by Lemma~\ref{l-propsafe} the propagation of the colouring can be done in polynomial time. Further, we consider $O(n^2)$ branches and for each branch, we apply a minimum cut algorithm. This can be done, for example, in $O(n^2|E|)$, using Dinic's algorithm~(\cite{Di70}).
    Thus, our algorithm runs in polynomial time.
\end{proof}

\noindent
For the proof of Theorem~\ref{T-P6+P4free}, we need the following lemma and apply Theorem~\ref{T-P6}. 

 \begin{lemma}\label{L-P4free}
     Let $G$ be a connected $P_4$-free graph on $n$ vertices. Then, $G$ has at most $2n$ different valid red-blue colourings.
 \end{lemma}

\begin{proof}

    It is well-known (see e.g.\ Lemma $2$ in~\cite{KP20}) that every connected $P_4$-free graphs has a spanning complete bipartite subgraph $K_{k,l}$ for some integers $1 \leq k \leq l$. 
    By Observation~\ref{basic-observations}, we get that $K_{k,l}$ is monochromatic if $k\geq 2$ and $l\geq 3$. If $k = 1$, that is, $G$ has a spanning star, there are at most $2n$ options to colour the graph ($2$ options to colour the centre vertex and $n$ options to choose at most one vertex with a colour different from that of the centre vertex).
    Otherwise, we get that $k = 2$ and $l=2$, that is, $G$ has a spanning $C_4$. It is easy to verify that there are at most $6 \leq 2n$ valid options to colour a $C_4$.
    Therefore, $G$ has at most~$2n$ different valid red-blue colourings.
\end{proof}

\begin{restatable}{theorem}{thmpsixpfourfree}
  \textsc{Minimum Matching Cut} is polynomial-time solvable for $(P_6+P_4)$-free graphs.
  \label{T-P6+P4free}
\end{restatable}

    \begin{proof}
Let $G = (V,E)$ be a $(P_6+P_4)$-free graph. If $G$ is $P_6$-free, we can determine whether $G$ has a minimum matching cut using Theorem~\ref{T-P6}. Thus, we may assume that $G$ is not $P_6$-free. Let $P$ be an induced $P_6$ of $G$. Note that we can find $P$ in $O(n^6)$.
We apply Observation~\ref{basic-observations} and search for a minimum red-blue colouring of $G$.

We branch over all possible colourings of $P$ and $N(P)$. There are at most $O(n^{6})$ such colourings since there is a constant number of colourings of $P$ and every vertex of $P$ has at most one neighbour of the opposite colour. If the colouring of $P$ and $N(P)$ is monochromatic, say all vertices are blue, then we branch over all $O(n)$ options of colouring one of the uncoloured vertices red. 
Let $B$ be the set of blue vertices of $G$ and let $R$ be the set of red vertices of $G$. We propagate the $(R,B)$-colouring using rules R1--R5. If we get a contradiction, we discard the branch.

    Since each of the six vertices of $P$ has at most one neighbour of the opposite colour, there are at most $6$ red and $6$ blue components in $G$.
    By Lemma~\ref{l-propconnectivity} this still holds after the propagation. 
Let $Z$ be the set of uncoloured vertices of $G$. Note that $G[Z]$ is $P_4$-free.
We make a first observation.

    \smallskip

        \clapp{\ref{T-P6+P4free}}{1}
         Let $v$ be a coloured vertex, say $v$ is blue. Suppose $v$ is adjacent to $w \in Z$. We colour $w$ and propagate the colouring. Regardless of the colour of $w$, we did not create an additional blue or red component.
        \begin{claimproof} 
            If $w$ is coloured blue, then, since $w$ is adjacent to the blue vertex $v$, we did not create a new connected component.
            
            Now assume that $w$ is coloured red. Let $C$ be the connected component of~$G[Z]$ containing~$w$. Then, since $v$ is blue, we colour all uncoloured neighbours of~$w$ red. 
            Note that by R4, the component $C$ has at least one vertex adjacent to a vertex in a red component.
            Let $p$ be such a vertex.
            Note that since $C$ is $P_4$-free, we have $\diam(C)\leq 2$.
            Therefore, either $w$ is adjacent to $p$ or they have a common neighbour.
            Since all neighbours of $w$ in~$C$ are coloured red, $p$ has two red neighbours and is therefore coloured red as well.
           Thus, $w$ is connected by a red path to a red component of $G$ and we did not create a new coloured component.
        \end{claimproof}

   \paragraph*{Blue Coast Processing}
    We repeat the following steps until we hit in Step~2 the same blue component a second time. Note that this requires at most $7$ iterations, since $G$ has at most $6$ blue components. 
    In the case where no connected component of size at least~$2$ remains in $G[Z]$, we apply Lemma~\ref{L-Indep Set} and either discard the branch or found a valid red-blue colouring.
    \begin{enumerate}
        \item Pick a connected component of $G[Z]$ of size at least~$2$.
        \item Pick a blue neighbour $w$ of this component.
        \item Let $s_1,s_2 \in C$ be such that $s_1 w, s_1s_2 \in E$. Let $q_1$ (resp.~$q_2$) be the only red neighbour of $s_1$ (resp.~$s_2$) and $b_2$ the only blue neighbour of $s_2$, if they exist. 
        \item Branch over all $O(n^5)$ colourings of $C\cup N(w)\cup N(q_1)\cup N(q_2)\cup N(b_2)$.
    \end{enumerate}

    \noindent
    Let $C$ be the currently chosen connected component of $G[Z]$. Note  that by~R4, every connected component of $G[Z]$ has a blue neighbour. Let $w$ be a blue neighbour of $C$ in the blue component $B_w$. Suppose first that we have not chosen $B_w$ before.
    Then $s_1$ and $s_2$ can each have at most one neighbour of the same colour, otherwise they would have been coloured by R2. Further, $C$ is $P_4$-free and has, by Lemma~\ref{L-P4free}, at most $O(n)$ red-blue colourings.

    Suppose now that we picked the blue component $B_w$ before, using a vertex~$v$. That is, there was a connected component $C'$ of $G[Z]$ with a blue neighbour~$v$ in~$B_w$. Let $s_1', s_2', q_1', q_2'$ and $b_2'$ be the vertices and neighbours of $C'$ considered while processing $C'$.
    Note that if there are several vertices $w \in B_w$ which are adjacent to $C$, we may assume that $\dist_B(v,w)$ is minimized among the options for~$w$.
    Further, $v \neq w$, since the neighbours of $v$ were coloured when processing $C'$ and $w$ has an uncoloured neighbour.
    Let now $s_1,s_2 \in C$ such that $s_1$ is adjacent to both $w$ and $s_2$. Let $v \, t_1 \, \ldots \, t_k \, w$ with $t_1,\ldots,t_k \in B_w$ be a shortest blue path between $v$ and $w$. Note that this implies that the path is an induced~$P_{k+2}$.

    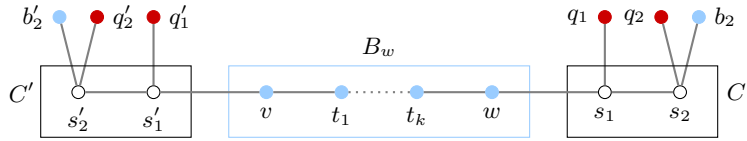
\begin{figure}
        \centering
        \begin{tikzpicture}
\begin{scope}
    \node[evertex, label=below:$s_2'$](s2p) at (0,0){};
    \node[evertex, label=below:$s_1'$](s1p) at (1,0){};

    \node[bvertex, label = left:$b_2'$](b2p) at (-0.25, 1){};
    \node[rvertex, label = right:$q_2'$](q2p) at (0.25,1){};
    \node[rvertex, label = right:$q_1'$](q1p) at (1,1){};

    \draw[edge](s1p) -- (s2p);
    \draw[edge](s1p) -- (q1p);
    \draw[edge](s2p) -- (q2p);
    \draw[edge](s2p) -- (b2p);

    \draw[] (-0.5, -0.6) rectangle (1.5,0.35);
    \node[] (c) at (-0.75,0){$C'$};
    
\end{scope}

\begin{scope}[shift = {(2.5,0)}]
    \node[bvertex, label = {below:$v$}](t0) at (0,0){};
    \node[bvertex, label = below:$t_1$](t1) at (1,0){};
    \node[bvertex, label = below:$t_{k}$](tk1) at (2,0){};
    \node[bvertex, label = {below:$w $}](tk) at (3,0){};

    \draw[edge](t0) -- (t1);
    \draw[edge, dotted](t1) -- (tk1);
    \draw[edge](tk1) -- (tk);

    \draw[lightblue] (-0.5,-0.6) rectangle (3.5,0.35);
    \node[] (bw) at (1.5, 0.6){$B_w$};

\end{scope}

\begin{scope}[shift = {(8,0)},  xscale = -1]
    \node[evertex, label=below:$s_2$](s2) at (0,0){};
    \node[evertex, label=below:$s_1$](s1) at (1,0){};

    \node[bvertex, label = right:$b_2$](b2) at (-0.25, 1){};
    \node[rvertex, label = left:$q_2$](q2) at (0.25,1){};
    \node[rvertex, label = left:$q_1$](q1) at (1,1){};

    \draw[edge](s1) -- (s2);
    \draw[edge](s1) -- (q1);
    \draw[edge](s2) -- (q2);
    \draw[edge](s2) -- (b2);

    \draw[] (-0.5, -0.6) rectangle (1.5,0.35);
    \node[] (c) at (-0.75,0){$C$};
    
\end{scope}

    \draw[edge](s1p) -- (t0);
    \draw[edge](s1) -- (tk);

\end{tikzpicture}
        \caption{The connected components $C$ and $C'$ adjacent to $B_w$ as considered in Claim~\ref{T-P6+P4free}.2.}
        \label{f-p4p6}
    \end{figure}

    \clapp{\ref{T-P6+P4free}}{2}
     There is an induced $P_6$ in $G[s_1, s_2, s_1', s_2', v, t_1, \dots, t_k, w]$. We may colour this path and its neighbourhood in polynomial time.
    \begin{claimproof}
        Note first that, due to R3, $s_2'$ is not adjacent to $v$ and~$s_2$ is not adjacent to $w$.
        Additionally, $s_i'$ and $s_j$, with $i,j\in \{1,2\}$, are not adjacent, otherwise $s_j$ would have been in $C'$.
        Further, $s_j$ cannot be adjacent to~$v$, since the neighbours of $v$ have been coloured while processing $C'$.
        If $w$ was uncoloured before processing $C'$, it cannot be adjacent to $s_i'$, otherwise $s_j$ would have been in $C$. If $w$ was blue before processing $C'$, then since $v \neq w$ and $s_1'$ had at most one blue neighbour, $w$ is not adjacent to $s_1'$. Also, $w$ is not adjacent to $s_2'$ since otherwise $w = b_2'$ and therefore, $s_1$ was coloured.
        Thus, we know that $G[v,s_1', s_2', w, s_1, s_2]$ is an induced $2P_3$. However, there might be edges from $t_\ell, \ell  \in \{1,\dots, k\}$ to $s_1', s_2', s_1$, and~$s_2$. 

        If there was an edge $t_\ell s_1$ or $t_\ell s_2$, for $\ell \in \{1,\dots, k\}$, then we would have chosen $t_\ell$ instead of $w$ due to the shorter distance from $v$ to $t_\ell$.
        Thus, there might only be edges from $t_\ell, \ell  \in \{1,\dots, k\}$ to $s_1', s_2'$.
        If there is no such edge, then $s_2'\, s_1'\, v\, t_1\, \dots\, t_{k}\, w\, s_1\, s_2$ is an induced path of length at least $6$ and we branch over all $O(n^5)$ colourings of $C\cup N(w)\cup N(q_1)\cup N(q_2)\cup N(b_2)$.
        Otherwise, let $\ell \in \{1,\dots, k\}$ be the largest integer such that one of $s_1'$ and $s_2'$ is adjacent to~$t_\ell$.
        
        If exactly $s_1'$ is adjacent to $t_\ell$, then $P' = s_2'\,s_1'\, t_\ell\, \dots\, t_{k}\, w\, s_1\, s_2$ is an induced path. If exactly $s_2'$ is adjacent to $t_\ell$, then $P' = s_1'\,s_2'\, t_\ell\, \dots\, t_{k}\, w\, s_1\, s_2$ is an induced path. In both cases, $P'$ contains an induced $P_6$.
        Note that $P'$ has length at most~$10$, since a $P_{11}$ contains an induced $P_6+P_4$. Thus, $k-\ell+1\leq 5$.
        We branch over all $O(n^{10})$ colourings of $C \cup N(t_\ell) \cup \dots \cup N(t_{k}) \cup N(w) \cup N(q_1) \cup N(q_2) \cup N(b_2).$

        Otherwise, both $s_1'$ and $s_2'$ are adjacent to $t_\ell$. Since $s_1', s_2', t_\ell$ form a triangle, by R3, they have all been coloured while processing $C'$.
        Note that $t_\ell$ is adjacent to $w$ since otherwise there is an induced $P_6$ contained in the induced path $s_1'\,t_\ell\,t_{\ell+1}\, \dots \, w\, s_1 \, s_2$.
        Therefore, $w$ was already blue when $C'$ was processed since otherwise $w$, $s_1$ and $s_2$ were in $C'$ and got coloured.
        It follows that $t_\ell$ is not adjacent to $v$; otherwise, it would have been adjacent to two blue vertices before processing $v$. Let $P' = v\,t_{1}\,\ldots\,t_{k}\,w\,s_1\,s_2$. Then, $P'$ contains an induced $P_6$. We branch over all $O(n^{11})$ colourings of \[C\cup N(w)\cup N(t_1)\cup \ldots\cup N(t_{k})\cup N(q_1)\cup N(q_2)\cup N(b_2). \]
    \end{claimproof}

        \smallskip\noindent
        Let $P'$ be the induced $P_6$ constructed in Claim~\ref{T-P6+P4free}.2. We may assume that $P'$ and its neighbourhood are coloured. We apply the propagation rules R1--R5 again. If we get a no-answer, we discard the branch, otherwise we continue.

        \clapp{\ref{T-P6+P4free}}{3}
         No uncoloured vertex and no red vertex of $G$ adjacent to an uncoloured vertex is adjacent to $P'$.
        \begin{claimproof}
            Note that we coloured all neighbours of $P'$. Thus, $P'$ has no uncoloured neighbours. Let $q$ be a red vertex adjacent to $P'$. Suppose for a contradiction that $q$ has an uncoloured neighbour $x$. Then, $q$ is not adjacent to a blue vertex of $P'$, otherwise $x$ would have been coloured. Suppose that $q$ is adjacent to $s_i$, for $i \in \{1,2\}$; the case of $s'_i$ works the same way. If $q$ was uncoloured when $s_i$ got coloured, then both $q$ and $x$ were in $C$ and therefore got coloured together with $s_i$. Otherwise, $q$ was $q_i$, the only red neighbour of $s_i$ and $x$ would have been coloured.
        \end{claimproof}

     \paragraph*{Red Coast Processing}

    Recall that $Z$ is the set of uncoloured vertices. If $G[Z]$ only consists of isolated vertices, we apply Lemma~\ref{L-Indep Set}. We either obtain a valid red-blue colouring in which case we remember its value or no such colouring exists and we discard the branch.

    So in the following, we may assume that $G[Z]$ has some connected component of size at least~$2$. After proving in Claim~\ref{T-P6+P4free}.4 that every red vertex has neighbours in at most one uncoloured component, we show that their structure is very restricted.

        \clapp{\ref{T-P6+P4free}}{4}
         Let $K$ be a connected component of $G[Z]$ of size at least $2$. Let $q\in R$ be a red neighbour of $K$. Every uncoloured neighbour of $q$ is in $K$.
        \begin{claimproof}
        Recall that by R4, every uncoloured component has a red neighbour.
            Suppose for a contradiction that there is a vertex $u\in Z$ which is adjacent to $q$ but not contained in $K$.
            Let $x_1,x_2 \in V(K)$ such that $q\,x_1, x_1\, x_2 \in E$. Then, $q\,x_2\notin E$, since otherwise $x_1$ and $x_2$ would be coloured by R3. Hence, $P'' = x_2\,x_1\,q\,u$ is an induced $P_4$. By Claim~\ref{T-P6+P4free}.3, $P' + P''$ is an induced $P_6+P_4$, a contradiction.
        \end{claimproof}

        \smallskip\noindent
        Let $K$ be a connected component of $G[Z]$ of size at least $2$.

        \smallskip\noindent
        \textbf{Case 1:} Suppose that there is a red vertex $q$ with at least two neighbours in $K$, say $x_1$ and~$x_2$. Then $x_1$ and $x_2$ cannot be adjacent since otherwise they would be coloured by R3. Since $x_1$ and $x_2$ are in the same connected component $K$ of $G[Z]$, they are connected by a path in $K$. If the shortest path in $K$ between them has length at least three, it is an induced~$P_4$, say $P''$, and by Claim~\ref{T-P6+P4free}.3 we obtain that $P' + P''$ is an induced $P_6+P_4$. Thus, the distance of $x_1$ and $x_2$ is exactly two and therefore there is a vertex $x_3 \in K$ such that $x_3$ is adjacent to both $x_1$ and $x_2$. Note that if $q$ was adjacent to $x_3$ then $x_1,x_2$ and $x_3$ would be red by R3.

        \clapp{\ref{T-P6+P4free}}{5}
            $K = \{x_1,x_2,x_3\}$
        \begin{claimproof}
            Suppose for a contradiction that there is an additional vertex $x_4$ in $K$. Without loss of generality, we may assume that $x_4$ is adjacent to $x_1$ or $x_3$. 
            Suppose first that $x_4$ is adjacent to $x_1$. Then, $x_4$ has to be adjacent to at least one of $q$ and $x_2$, since, by Claim~\ref{T-P6+P4free}.3, $x_2\,q\,x_1\,x_4$ cannot be an induced $P_4$. Adjacency to $q$ leads to a triangle $x_4\,q\,x_1$, implying that $x_1$ and $x_4$ are coloured by R3, a contradiction. Therefore, $x_4$ is adjacent to $x_2$. This implies that $K_{2,3}$ is a spanning subgraph of $K\cup \{q\}$. Thus, $K$ has to be monochromatic and coloured red by R5, a contradiction.
            
            Consider now the case where $x_4$ is not adjacent to $x_1$. Hence, $x_4$ is adjacent to~$x_3$. Then, $x_4$ is adjacent to $q$ or to both $x_1$ and $x_2$, since, by Claim~\ref{T-P6+P4free}.3, neither $q\,x_1\,x_3\,x_4$ nor $q\,x_2\,x_3\,x_4$ are induced $P_4$s. Note that by assumption $x_4$ is not adjacent to $x_1$. 
            If $x_4$ is adjacent to $q$, then $K\cup \{q\}$ has again $K_{2,3}$ as a spanning subgraph and by R5, we get a contradiction. Thus, in all cases, we get a contradiction, and the claim holds. 
        \end{claimproof}

        \smallskip\noindent
        From Claim~\ref{T-P6+P4free}.5 we got that $K = \{x_1, x_2, x_3 \}$. In the following, we analyse the blue neighbours of $K$. Let $y_i$ be the blue neighbour (if it exists) of $x_i$ for $i\in \{ 1,2,3\}$. Note that $y_1, y_2$ and $y_3$ are not necessarily different.
        Depending on the blue neighbours, we colour such a component~$K$ as follows:
        \begin{enumerate}
            \item  If $x_1$ and $x_2$ both have a blue neighbour, colour $K$ red.
            \item If exactly $x_1$ and $x_3$ have a blue neighbour, colour $x_2$ red and $x_1, x_3$ blue.
            \item If exactly $x_2$ and $x_3$ have a blue neighbour, colour $x_1$ red and $x_2, x_3$ blue.
            \item If exactly $x_1$ has a blue neighbour, colour $x_2$ red.
            \item If exactly $x_2$ has a blue neighbour, colour $x_1$ red.
            \item If exactly $x_3$ has a blue neighbour, colour $x_1$ red.
        \end{enumerate}

        \clapp{\ref{T-P6+P4free}}{6}
         $G$ has a minimum red-blue $(R,B)$-colouring before applying the aforementioned colouring rules if and only if it has one afterwards.
        \begin{claimproof}
            \textit{1.} If $x_3$ is coloured blue, then so are $x_1$ and $x_2$. Thus, $q$ has two blue neighbours, a contradiction. Hence, we can safely colour $x_3$ red. Then $x_1$ and $x_2$ have to be coloured red, too.\\
            \textit{2.-3.} If exactly $y_1$ (or $y_2$ by symmetry) and $y_3$ exist then there are $2$ ways of colouring~$K$. Either we colour $K$ red or we only colour $x_2$ red and $x_1$ and $x_3$ blue. Both colourings contribute $2$ to the value of the red-blue colouring. If we choose the latter colouring, $y_1$ and $y_3$ still have the possibility of having a red neighbour at another point in the colouring process. Since $q, x_1, x_2$ and $x_3$ do not have uncoloured neighbours outside of $K$, this rule is safe.\\
            \textit{4.-6.} 
            In all three cases, we can colour every vertex in $K$ red. This leads to a contribution of $1$ to the total value of the colouring and the blue neighbour is prevented from having a red neighbour outside of $K$.
            Alternatively, we can choose the following colourings.
            \begin{description}
                \item[4.] Colour $x_2$ red and $x_1$ and $x_3$ blue.
                \item[5.] Colour $x_1$ red and $x_2$ and $x_3$ blue.
                \item[6.] Colour $x_1$ red and $x_2$ and $x_3$ blue.
            \end{description}
            In each case, the colouring contributes $2$ to the total value and allows the blue neighbour of~$K$ to have a red neighbour outside of $K$.
            Note further, that in all three cases, there is one vertex ($x_2$ in Case 1, $x_1$ in Case 2 and 3) which is red in both possible colourings. Thus, we may safely colour this vertex red.
        \end{claimproof}

        \smallskip\noindent
        After the application of the above rules, we apply R1 and discard the branch if we get a no-answer. Otherwise, every component $K$ considered in this case got partially coloured and we are left with connected components of size~$2$ where both vertices have a private red neighbour and one vertex has a blue neighbour.
        This ends the case where $K$ has a red neighbour adjacent to at least two vertices in~$K$ and we proceed with the next case.
        
        \smallskip \noindent
       \textbf{Case 2:} Every red neighbour of $K$ is adjacent to exactly one vertex in $K$.           
       Suppose first that $K$ has exactly one red neighbour. Recall that $K$ has at least~$1$ blue neighbour, since otherwise $K$ would be red by R4. Thus, any colouring of~$K$ contributes at least $1$ to the total value of the red-blue colouring. Recall further that any red neighbour of $K$ has uncoloured neighbours only in $K$. Therefore, we colour $K$ blue, implying that it contributes exactly $1$ to the total value and does not affect the colouring of the other vertices.

       Consider now the case where $K$ has at least $2$ red neighbours $q_1,q_2$. Let $x_1,x_2\in K$ such that $q_1$ is adjacent to $x_1$ and $q_2$ is adjacent to $x_2$. Since $x_1$ and~$x_2$ are both in $K$, there is a path in $K$ connecting them. Thus, if $x_1$ and $x_2$ are not adjacent, $P'' = q_1\,x_1\,\dots\,x_2$ is an induced $P_4$ and thus, by Claim~\ref{T-P6+P4free}.3, $P' + P''$ is an induced $P_6 + P_4$, a contradiction. Hence, $x_1$ and $x_2$ are adjacent. Since for the same reason, $q_1\,x_1\,x_2\,q_2$ may not be an induced $P_4$ either, and any edge $x_iq_j$, $i,j \in\{1,2\}$ would result in a triangle with one coloured vertex (impossible by R3), there must be an edge between $q_1$ and~$q_2$.

        \clapp{\ref{T-P6+P4free}}{7}
           $K = \{x_1,x_2\}$.
       \begin{claimproof}
           Suppose for a contradiction that there is a vertex $x_3 \in V(K)$ adjacent to $x_1$. Since $q_2\,q_1\,x_1\,x_3$ cannot be an induced $P_4$, there is an edge from $x_3$ to $q_1$ or $q_2$, a contradiction since, by assumption, both $q_1$ and $q_2$ only have one uncoloured neighbour. The case where $x_3$ is adjacent to $x_2$ follows by symmetry. 
       \end{claimproof}

       \smallskip\noindent
       Let $y_1$ and $y_2$ be the unique, blue, not necessarily different neighbours of $x_1$ and $x_2$ (if they exist).
       If both $y_1$ and $y_2$ exist, there are two possibilities to colour $x_1$ and $x_2$.
       Either both red or both blue, in each case contributing $2$ to the total value of the colouring.
       If both are blue, $y_1$ and $y_2$ may have a red neighbour in another connected component. Recall that $q_1$ and $q_2$ do not have other uncoloured neighbours. Hence, we can safely colour $K$ blue.

       Thus, every component $K$ considered in this case either got coloured or is of size $2$ where both vertices have a private red neighbour and exactly one vertex has a blue neighbour.

       \paragraph*{Colouring the remaining vertices}
        
        Recall that $Z$ is the set of uncoloured vertices. 
        Note that we are left with two types of connected components of $G[Z]$.
        We have connected components of size $2$ with neighbours as depicted in Figure~\ref{fig-t12} (left) and connected components of size $1$ with a red and a blue neighbour, see Figure~\ref{fig-t12} (right).

        Let $W = Z\cup N(Z)$ and let $H$ be the graph with $V(H) = W$ and $E(H) = \{uu' \in E | u\in Z, u' \in W \}$. 
        That is, $H$ contains all edges of $G$ with at least one uncoloured endvertex. 

        Note that if $H$ consists of several connected components, their colourings do not influence each other. Therefore, we may consider the components separately and we assume in the following that $H$ is connected.

        \begin{figure}
            \centering
            \begin{tikzpicture}
    \begin{scope}[yscale = 0.7, xscale = 0.5]
        \node[rvertex](r1) at (-2,0){};
        \node[rvertex](r2) at (-2,1){};
        \node[bvertex](b1) at (2,0) {};
        \node[evertex](v1) at  (0,0) {};
        \node[evertex](v2) at  (0,1){};

        \draw[tedge] (r1) -- (v1);
        \draw[tedge] (r2) -- (v2);
        \draw[tedge] (v1) -- (v2);
        \draw[tedge] (v1) -- (b1);
    \end{scope}

    \begin{scope}[yscale = 0.7, xscale = 0.5, shift = {(7,0)}]
        \node[rvertex](r1) at (-2,0){};
        \node[evertex](v1) at (0,0){};
        \node[bvertex](b1) at (2,0){};

        \draw[tedge] (r1) -- (v1);
        \draw[tedge] (v1) -- (b1);
        
    \end{scope}
\end{tikzpicture}
            \caption{A component of size $2$ (left) and a component of size $1$ (right).}\label{fig-t12}
        \end{figure}
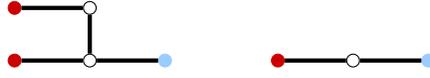

        Recall that by Claim~\ref{T-P6+P4free}.4, every red vertex which is adjacent to a component~$K$ of $G[Z]$ of size~$2$ has exactly one uncoloured neighbour which is inside $K$.
        We say that a coloured vertex $v$ has \emph{property $Q$} if every uncoloured neighbour of $v$ has to be coloured with the same colour as $v$.

        \clapp{\ref{T-P6+P4free}}{8}
           Suppose that a coloured vertex $v$ satisfies $Q$, then we can colour the whole component of $H$ containing $v$ by propagation or conclude that $G$ has no red-blue colouring which extends the current partial colouring.
        \begin{claimproof}
        Let $v$ be a coloured vertex satisfying $Q$. Let $T$ be the set of uncoloured neighbours of~$v$. Suppose that there is a vertex $w$ in $T$ with a second coloured neighbour $u$. By R2, $u$ has the opposite colour of $v$. Thus, if we colour $w$ according to $Q$, that is, with the colour of~$v$, then $u$ has a neighbour of the other colour. Hence, all other neighbours of $u$ have to have the same colour as $u$ and $Q$ is satisfied for $u$.

        Thus, if a vertex $v$ satisfies $Q$, then, after propagation, the coloured vertices at distance~$2$ of $v$ satisfy $Q$ and the uncoloured neighbours of $v$ are coloured. We continue the propagation exhaustively and either obtain a contradiction to the validity of the colouring and discard the branch, or the whole component is coloured.        
        \end{claimproof}

        \clapp{\ref{T-P6+P4free}}{9}
        Let $T$ be a connected component of $G[Z]$ of size $2$. If we colour $T$ red and propagate the colouring, then either we obtain a valid red-blue colouring of~$H$ or we conclude that $T$ is be blue in any valid red-blue colouring of $H$.
        \begin{claimproof}
            Let $T$ be a connected component of $G[Z]$ of size $2$. We colour the vertices of $T$ red. Then, the neighbours of $T$ satisfy property $Q$. For the two red neighbours of $T$, this follows immediately from Claim~\ref{T-P6+P4free}.4. For the blue neighbour of $T$, this follows since $T$ gets coloured red.
            The claim follows from application of Claim~\ref{T-P6+P4free}.8.
        \end{claimproof}

        \smallskip\noindent
        Let $k = \lvert \{ C : C \textrm{ is a connected component of size $2$ of $G[Z]$} \} \rvert$ be the number of connected components of size~$2$ in $G[Z]$. If $k=0$, we conclude by applying Lemma~\ref{L-Indep Set}.
        
        Otherwise, let $T\in H$ be a connected component of size $2$. 
        First, we colour $T$ red. By Claim~\ref{T-P6+P4free}.9, we either obtain in polynomial time a valid red-blue colouring and remember its value or we get that no such colouring exists and discard the branch.
        Second, we colour $T$ blue and propagate the resulting colouring. Then only $k-1$ components of size $2$ remain and we can repeat the same process and try to colour in red the uncoloured vertices of some component of size $2$.

    If at any point in our algorithm we discard a branch, we consider the next. For every valid red-blue colouring which we obtain, we remember its value and output the colouring with the minimum value. The correctness of our algorithm follows from its description. In the following, we analyse its run-time.

    By Lemma~\ref{l-propsafe}, the propagation of a colouring can be done in polynomial time. We consider $O(n^{48})$ branches, each of which can be processed in polynomial time. Therefore, the run-time of our algorithm is polynomial.
\end{proof}

\noindent
For our next theorem, we apply Theorem~\ref{T-Pk} twice, leading to a $P_4$-free dominating subgraph~$K$ of a dominating subgraph. Being $P_4$-free has the advantage that all but potentially one vertex of $K$ are coloured the same, say red. After some branching we are able to conclude by applying Lemma~\ref{L-monodom}.

\begin{theorem}\label{t-p8}
    \minmc{} is polynomial-time solvable for $P_8$-free graphs.
\end{theorem}

\begin{proof}
    Let $G = (V,E)$ be a $P_8$-free graph. We assume that $G$ is connected. 
    We apply Observation~\ref{basic-observations} and search for a minimum red-blue colouring of $G$. Theorem~\ref{T-Pk} states that $G$ has either a dominating $P_6$ or a dominating connected $P_6$-free subgraph and that such a subgraph can be found in polynomial time. If $G$ has a dominating $P_6$, we have a dominating set of bounded size and apply Lemma~\ref{l-smalldomset}. 
    In polynomial time, we either find a minimum red-blue colouring or conclude that no such colouring exists.

    Hence, we may assume that $G$ has a dominating set $D$ such that $G[D]$ is connected and $P_6$-free. By applying Theorem~\ref{T-Pk} again, we obtain that $G[D]$ either has a dominating $P_4$ or a connected dominating $P_4$-free subgraph, say it is $K$. 
    Note that a connected $P_4$-free subgraph has a complete bipartite graph $K_{r,s}$ as a spanning subgraph, for $r,s \geq 1$.
    
    We distinguish between two cases. First, we consider the case where $|K|\leq 4$. Note that this covers both the case of a dominating $P_4$ and of a dominating $P_4$-free subgraph where $r + s \leq 4$.
    Second, we consider the case where $|K|\geq 5$. The arguments are  similar in both cases, however in the first case, we can colour $N[V(K)]$ completely. Due to the fact that the dominating subgraph $K$ may be large in the second case, we cannot assume that the neighbourhood of $K$ is fully coloured. This leads to a significant increase in the complexity of the arguments needed and in particular makes it necessary to add a second recursive step.

\smallskip
\noindent
    \textbf{Case 1. }\textit{$G[D]$ is dominated by a connected dominating set of size at most~$4$.} \\ 
Let $D'$ be a connected dominating set of $G[D]$ of size at most~$4$.
    We branch over all colourings of $D'$. We denote by $R$ (resp.~$B$) the set of red (resp.~blue) vertices. If $D'$ is monochromatic, we apply Lemma~\ref{l-ddmp} and either find a minimum red-blue $(R,B)$-colouring and remember its value or conclude that no such colouring exists and discard the branch.

    We now assume that $D'$ is not monochromatic. That is, $D'$ contains at least one red and one blue vertex. Note that there might be two vertices of the same colour in $D'$ that are not connected in $D'$. Say they are red. This can happen if $D'$ is a $P_4$ with red ends.
    However, in every optimal solution of \minmc{} there is exactly one red and one blue component. Since $G$ is $P_8$-free, there is a monochromatic path in $G$ connecting those two vertices in $D'$ consisting of at most~$5$ intermediate vertices. We therefore branch over all $O(n^5)$ options to choose this path. Hence, we have one red and one blue component.

    We branch over all $O(n^4)$ valid colourings of $N(D')$ and propagate the colouring. Let $Z$ be the set of uncoloured vertices of $G$.
    If every vertex of $Z$ is adjacent to a red vertex, then $R$ dominates $Z$ and we apply Lemma~\ref{L-monodom}. 
    If we get a valid red-blue $(R,B)$-colouring of $G$, we remember its value, otherwise we get that no such colouring exists and discard the branch.

    Hence, there exists a vertex $x \in Z$ which is not adjacent to any vertex in $R$.
    As $x$ is at distance at most~$2$ of $D'$ and the neighbourhood of $D'$ is coloured, there is a vertex $b \in B$ which is adjacent to both $x$ and a blue vertex of $D'$.
    Let $C$ be the connected component of $x$ in $G[Z]$.
    Due to the application of R4, we know that $C$ is adjacent to $R$. Hence, since $x$ is not adjacent to $R$, it has a neighbour $y \in C$, see Figure~\ref{f-p7-dom-p2}. 
    Let $\beta$ be a blue neighbour of $b$ in~$D'$. If there are several we take the one with the shortest distance to a red vertex in $D'$. Let further $\alpha$ be a red vertex in $D'$ that is closest to $\beta$ in $D'$.
    
    \begin{figure}[t]
        \centering
        \begin{tikzpicture}
\begin{scope}
    \node[evertex, label=below:$y$](s2p) at (0,0){};
    \node[evertex, label=below:$x$](s1p) at (1,0){};


    \draw[edge](s1p) -- (s2p);

    \draw[] (-0.5, -0.6) rectangle (1.5,0.35);
    \node[] (c) at (-0.75,-0.4){$C$};
    
\end{scope}

\begin{scope}[shift = {(2.5,0)}]
    \node[bvertex, label = {below:$b$}](t0) at (0,0){};
    \node[bvertex, label = below:$\beta$](t1) at (1,0){};
    \node[rvertex, label = below:$\alpha$](tk1) at (2,0){};
    
    \node[rvertex, label = above:$r$](b) at (3,0){};
    \node[evertex, label = above:$w$](w) at (4,0){};

    \draw[edge](t0) -- (t1);
    \draw[edge](t1) -- (tk1);

    \draw[edge](tk1) -- (b);
    \draw[edge](b) -- (w);

    \draw[] (0.5,-0.6) rectangle (2.5,0.35);
    \node[] (bw) at (2.75, -0.4){$D'$};

\end{scope}

    \draw[edge](s1p) -- (t0);

\end{tikzpicture}
        \caption{$D'$ together with a component $C$ of size at least~$2$.}
        \label{f-p7-dom-p2}
    \end{figure}
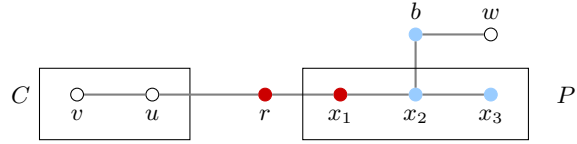

    \clapp{\ref{t-p8}}{2}
       $P' = y \, x \, b \, \beta \, \dots \, \alpha$ is an induced path.

    \begin{claimproof}
        To see this, note first that $\beta \, \dots \, \alpha$ is an induced path since it is a shortest path in $D'$. 
        Since $x$ and $y$ are both uncoloured, they cannot be adjacent to $D'$ and at most one can be adjacent to $b$.
        Further, if $b$ was adjacent to a red vertex in $D'$, $x$ would have been coloured blue. 
        Also, $b$ is not adjacent to any blue vertex on the path $\beta \, \dots \, \alpha$, by our choice of $\beta$.
        Thus, $P'$ is an induced path.
    \end{claimproof}

    \clapp{\ref{t-p8}}{3} 
    If $y$ is coloured blue, we can find in polynomial time a minimum red-blue $(R,B)$-colouring of $G$ or conclude that no such colouring exists.

    \begin{claimproof}
    We colour $y$ blue and propagate the colouring. Hence, by~R2, $x$ is blue as well. We branch over the $O(n^3)$ colourings of the neighbours of $b,x$ and $y$.
    If $Z$ is dominated by $B$ we apply Lemma~\ref{L-monodom}. We either obtain a red-blue $(R,B)$-colouring and remember its value or conclude that no such colouring exists and discard the branch.
    
    Let $w\in G[Z]$ be not adjacent to $B$. As $w$ is at distance~$2$ of $D'$, there is a vertex $r$ adjacent to both $D'$ and $w$, see Figure~\ref{f-p7-dom-p2}. We know that $r\in R$, since $N(D')$ is coloured and $r$ is not blue. Further, $r$ cannot be adjacent to any vertex in $B$, otherwise $w$ would have been coloured. If $r$ is adjacent to $\alpha$, then $y \, x \, b \, \beta \, \dots \, \alpha \, r \, w$ is an induced path containing at least two vertices in $D'$, so at least an induced $P_7$.
    Since $w$ is uncoloured and not adjacent to~$B$, there is an uncoloured neighbour $w'$ of $w$. Note that $w'$ is neither adjacent to any of $r, b, x,y$ nor to a vertex in $D'$, since otherwise it would be coloured. Hence, $y \, x \, b \, \beta \, \dots \, \alpha \, r \, w \, w'$ contains an induced $P_8$, a contradiction. 
    Thus, the claim follows.
    \end{claimproof}

    \noindent
    We first check whether colouring $y$ blue leads to a valid red-blue colouring, in which case we remember its value.
	Then, we colour $y$ red and propagate the colouring.
    Now the number of uncoloured vertices decreased and we check again whether $R$ dominates $Z$ and otherwise find another connected component of $G[Z]$ of size at least~$2$, try to colour the vertex in the role of $y$ blue and repeat the whole process.
    Whenever we discard a branch, we consider the next. Whenever we obtain a valid red-blue colouring, we remember its value. In every iteration, the number of uncoloured vertices decreases and we only recurse in the case where $y$ is red. Thus, we recurse at most $O(n)$ times, and for each step in the recursion, we consider $O(n^3)$ sidebranches in the case where $y$ is blue. Together with the initial branching to colour the neighbours of $D'$ and a potential branching to connect the red vertices in $D'$, we consider in total $O(n^{13})$ branches.
    Since the application of Lemma~\ref{L-monodom} and the propagation can be done in polynomial time, our algorithm runs in polynomial time in the case where $G[D]$ is dominated by a connected dominating set of size at most~$4$.

\smallskip
\noindent
\textbf{Case 2. }\textit{$G[D]$ is dominated by a subgraph of at least~$5$ vertices.}\\
Let $K$ be a connected subgraph of $G$ which dominates $G[D]$. Since we are not in Case~1, $|V(K)| \geq 5$. Further, $K$ has a spanning $K_{r,s}$ subgraph that is either a star, or $r \geq 2, s \geq 3$. In the latter case, $K$ is monochromatic in every valid red-blue colouring of $G$. If $K$ is a star, it may be monochromatic or bichromatic.

We now construct a partial colouring, depending on whether $K$ is monochromatic or a bichromatic star. The arguments will then be the same until mentioned otherwise.

Consider first the case where the spanning $K_{r,s}$ subgraph of $K$ is a bichromatic star. 
Let~$\alpha$ be the centre of the star, without loss of generality we colour $\alpha$ red. We branch over all $O(n)$ possibilities to colour a neighbour of $\alpha$ blue. We call the blue neighbour of $\alpha$ $\beta$.

Consider now the case where $K$ is monochromatic. We colour $K$ red. Assume first that all neighbours of $K$ are coloured red. 
Then, $N(K) \cup K$ dominates~$G$, we branch over all $O(n)$ options to colour some vertex of $G$ blue and apply Lemma~\ref{L-monodom}. We either obtain a valid red-blue colouring of $G$ and remember its value, or conclude that no such colouring exists and discard the branch.
Else, we branch over all $O(n)$ possibilities to colour a neighbour of $K$ blue. Let $\alpha$ be the vertex in $K$ with a blue neighbour and $\beta$ the blue neighbour of $\alpha$. 

In both cases, we propagate the colouring using R1--R5.
If we get a no-answer, we discard the branch, otherwise, we continue.
Let $R$ (resp.~$B$) be the set of red (resp.~blue) vertices and let $Z$ be the set of uncoloured vertices. Observe that both $R$ and $B$ are connected in $G$.

If every connected component of $G[Z]$ is of size~$1$, we apply Lemma~\ref{L-Indep Set}. We either obtain a minimum red-blue $(R,B)$-colouring of $G$ and remember its value, or conclude that no such colouring exists and discard the branch.
Hence, we may assume that there is a connected component $C$ of $G[Z]$ of size at least~$2$.
Recall that by R4, $C$ has to be adjacent to $B$. Let $x,y \in C$ such that $x$ and $y$ are adjacent and $x$ is adjacent to a blue vertex $b \in B$, see Figure~\ref{f-p8-case2}. 

Let $b \, p_1 \, \dots \, p_\ell \, \beta$ be the shortest path in $B$ from $b$ to $\beta$. Note first that due to R3, not both of $x$ and $y$ may be adjacent to the same vertex in $p_1, \dots, p_\ell$. This allows us to assume that neither $x$ nor $y$ is adjacent to any of $p_1, \dots, p_\ell$. Else, if $p_j$, for $j \in \{1, \dots, \ell\}$, is adjacent to $x$ or $y$, after potentially changing the role of $x$ and $y$, we replace $b$ by $p_j$.
We claim that the path $y \, x\, b \, p_1 \, \dots \, p_\ell \, \beta \, \alpha$ is induced. Note that $b \, p_1 \, \dots \, p_\ell \, \beta$ is an induced path since it was chosen as a shortest blue path from $b$ to $\beta$. By construction, $x$ and $y$ are not adjacent to any of $p_1, \dots, p_\ell$ and $y$ is not adjacent to $b$. Since $\beta$ and $\alpha$ are neighbours of different colour, their other neighbours are coloured alike and neither $x$ nor $y$ can be adjacent to them. In addition, $\alpha$ cannot have a second blue neighbour in $b, p_1, \dots, p_\ell$. Hence, the path is induced. Note that $\ell \leq 2$, since otherwise $y \, x\, b \, p_1 \, \dots \, p_\ell \, \beta \, \alpha$ contains an induced $P_8$, a contradiction.

\begin{figure}
    \centering
    \begin{tikzpicture}
\begin{scope}
    \node[bvertex, label=below:$y$](v1) at (0,0){};
    \node[bvertex, label=below:$x$](v2) at (1,0){};

    \node[bvertex, label = {below:$b$}](v3) at (2,0){};
    \node[bvertex, label = below:$\beta$](v4) at (3,0){};
    \node[bvertex, label = below:$b'$](v5) at (4,0){};
    \node[bvertex, label = below:$x'$](v6) at (5,0){};
    \node[bvertex, label = below:$y'$](v7) at (6,0){};

    \draw[edge](v1) -- (v2);
    \draw[edge](v2) -- (v3);

    \draw[edge, dotted](v3) -- (v4);
    \draw[edge](v4) -- (v5);
    \draw[edge](v5) -- (v6);
    \draw[edge](v6) -- (v7);

    \node[rvertex, label = {[label distance=-0.4em]180:$q_y$}](r1) at (0,1){};
    \node[rvertex, label=left:$q_x$](r2) at (1,1){};
    \node[rvertex, label = {left:$q_b$}](r3) at (2,1){};
    \node[rvertex, label = above:$\alpha$](r4) at (3,1){};
    \node[rvertex, label = above:$q_b'$](r5) at (4,1){};
    \node[rvertex, label = above:$q_x'$](r6) at (5,1){};
    \node[rvertex, label = above:$q_y'$](r7) at (6,1){};

    \draw[edge](v1) -- (r1);
    \draw[edge](v2) -- (r2);
    \draw[edge](v3) -- (r3);
    \draw[edge](v4) -- (r4);
    \draw[edge](v5) -- (r5);
    \draw[edge](v6) -- (r6);
    \draw[edge](v7) -- (r7);

    \draw[edge, dotted](2.4,1) -- (r4);

    \node[rvertex, label = left:$w$](w) at (1,2){};
    \node[evertex, label = right: $\gamma$] (gam) at (2.5,2){};

    \draw[dotted, edge] (1,1.3) -- (1,1.6);
    \draw[edge] (w) -- (gam);

    \draw[] (-0.4, -0.6) rectangle (1.4,0.3);
    \node[] (c) at (-0.75,-0.4){$C$};

    \draw[] (-0.4, 0.6) rectangle (2.4,1.3);
    \node[] (c) at (-0.75,0.8){$Q$};

    \draw[] (0.3, 1.6) rectangle (1.7,2.3);
    \node[] (c) at (0.05,1.8){$S$};
    
\end{scope}

\end{tikzpicture}
    \caption{An illustration of the structure in the proof of Theorem~\ref{t-p8}.}
    \label{f-p8-case2}
\end{figure}

\begin{myclaim}
    If $y$ is coloured blue we can find in polynomial time a minimum red-blue $(R,B)$-colouring of $G$ or conclude that no such colouring exists.
\end{myclaim}
\begin{claimproof}
    We colour $y$ blue and propagate the colouring. Hence, by R2, $x$ is blue as well. We branch over the $O(n^5)$ colourings of the neighbours of $b, x, y, p_1$ and $p_2$ (if they exist) and propagate the colouring using R1--R5.

    Let $Q \subseteq R$ be the set of red neighbours of $b,x,y,p_1$ and $p_2$ (if they exist). Note that due to R2, we have that $|Q| \leq 5$.
    Note further that all neighbours of $Q$ except $b,x, y, p_1$ and $p_2$ are red by propagation.
    Recall that $N(Q)$ and $N(\alpha)$ are already coloured. Hence, since $G$ is $P_8$-free, any such path needs at most $3$ newly coloured vertices, leading to $O(n^{15})$ branches.

    Let $S \subseteq R$ be the set of vertices of $R \setminus Q$ that are connected (by only using red vertices) to $\alpha$ only via $Q$.
    Note that $S$ might be empty.
    We partition $G[Z]$.
    Let $C_1 \subseteq Z$ be the set of uncoloured vertices which do not have a red neighbour or whose red neighbour is in $S$. 
    Let $C_2 \subseteq Z$ be the set of uncoloured vertices with a red neighbour in $R \setminus (Q \cup S)$.
    Note that this is indeed a partition of $Z$, since the neighbours of $Q$ have already been coloured.

    Suppose for a contradiction that a vertex $u \in C_2$ has an uncoloured neighbour $v$. Let $r$ be the unique red neighbour of $u$. We claim that $v\, u \, r \, \dots \, \alpha \, \beta \, \dots \, b \, x \, y$ contains an induced $P_8$. To see this, recall that the neighbourhoods of $\alpha, \beta, \dots, b, x, a$ are coloured, $\beta$ is not adjacent to any neighbour of $\alpha$ and $v,u$ are uncoloured. Hence, we get a contradiction. This implies in particular that $G[C_2]$ is an independent set and that $C_2$ is dominated by $R$.

    We check whether $R$ dominates $Z$. If this is the case, we apply Lemma~\ref{L-monodom}. If we get a valid red-blue $(R,B)$-colouring of $G$, we remember its value, otherwise we get that no such colouring exists and discard the branch.
    
    Hence, we may assume that there is an uncoloured vertex $u$ in $C_1$ which does not have a red neighbour.
    Since $K$ dominates $G[D]$ and $D$ dominates $G$, $u$ is at distance at most $2$ of~$K$.
    Let $v$ be a vertex on a shortest path from $u$ to $K$, that is, $v$ is adjacent to both $u$ and~$K$. 
    Suppose that $v$ is adjacent to a red vertex in $K$.
    We distinguish three cases.
    \begin{itemize}
        \item If $v$ is red, then $u$ has a red neighbour, a contradiction.
        \item If $v$ is blue, then $u$ is coloured by propagation, a contradiction.
        \item If $v$ is uncoloured, then $v\in C_2$. The observation above implies that $v$ has no uncoloured neighbour, a contradiction.
    \end{itemize}
    So in all three cases, we get a contradiction.
    Hence, every vertex in $C_1$ without a red neighbour is at distance~$2$ from a blue vertex in $K$.
    If $K$ is monochromatic, there is no blue vertex in $K$ and thus, every vertex in $C_1$ has a red neighbour. Hence, $R$ is dominating and we apply Lemma~\ref{L-monodom}. If we get a valid red-blue $(R,B)$-colouring of $G$, we remember its value, otherwise we get that no such colouring exists and discard the branch. 
    Thus, the claim follows in the case where $K$ is monochromatic.

    \smallskip
    We are left with the case where $K$ is a bichromatic star.
    Since $\beta$ is the only blue vertex in $K$, every vertex in $C_1$ without a red neighbour is at distance~$2$ from $\beta$ and has a blue neighbour in $N(\beta)$.
    Hence, we may assume that there is an uncoloured vertex $x'$ with a blue neighbour $b'$, which is not adjacent to a red vertex. Further,  $b'$ is adjacent to $\beta$. Let $y'$ be an uncoloured neighbour of $x'$, see Figure~\ref{f-p8-case2}. Note that $y'$ exists, since every component of $G[Z]$ is adjacent to a red vertex. We require a second side branch.

    \smallskip \noindent
    \textit{Suppose $y'$ is coloured blue.}
    We propagate the colouring and obtain by R2 that $x'$ is blue as well.
    We branch over the $O(n^3)$ colourings of the neighbourhoods of $b', x', y'$. Let $q_b', q_x',q_y'$ be the red neighbours of $b', x', y'$, if they exist. 
    Observe first that $x'$ and $b'$ had no red neighbour before colouring $N(x')$ and $N(b')$, else $x'$ would have been coloured. Since $N(Q)$ is coloured, we get that in particular, $q_b', q_x', q_y' \notin Q$ and $q_b', q_x' \notin N(Q)$.
    Second, $Q \, \dots \, \alpha \, \beta \, b' \, x'$ is induced, where $Q \dots  \alpha$ is a red path from $Q$ to $\alpha$ avoiding $q_x'$ and $q_b'$. Note that such a path exists, since there was a red path from every vertex in $Q$ to $\alpha$ before choosing $x'$. It is induced since no neighbour of $b', x'$ was red before choosing $x'$.

    In the following we show that, whenever $S$ has an uncoloured neighbour $\gamma$, we find an induced $P_8$.    
    Let $\gamma$ be an uncoloured neighbour of $S$.
    Note that $\gamma$ is not adjacent to any of $\alpha, \beta, b', x', y', q_b', q_x', q_y'$ nor to any vertex in $Q$, otherwise it would be coloured.

    Let $w$ be the neighbour of $\gamma$ in $S$.
    Suppose there is a red path $w \, \dots \,\alpha$ not using any of $q_b', q_x',q_y'$. 
    We claim that $\gamma \, w\, \dots \, \alpha \, \beta \, b' \, x' \, y'$ contains an induced~$P_8$. 
    It contains at least~$8$ vertices since $w \in S$, and so any red path $w \, \dots \, \alpha$ passes via $Q$.
    We saw that $Q \, \dots \,\alpha \, \beta \, b' \, x'$ is induced. Since the chosen path from $w$ to $\alpha$ avoids $q_b', q_x', q_y'$, the whole path is induced.

    Hence, we are left with the case where every red path $w \, \dots \, \alpha$ contains at least one of $q_b', q_x',q_y'$.
    Since, $q_b', q_x', q_y' \notin Q$, we get that they are either on the path $w \, \dots \, Q$ or on the path $Q\, \dots \, \alpha$. Since for any vertex in $Q$ there was a red path $Q \, \dots \, \alpha$ and $b', x'$ had no red neighbour before choosing $x'$, we only consider the latter case for $q_y'$.

    Observe that if $q_y'$ lies on a red path $Q\, \dots \, \alpha$ and $q_b'$ (resp.~$q_x'$) in $S$, then $q_y' q_b' \notin E$ (resp.~$q_y'q_x' \notin E$).
    We now distinguish several cases based on which vertices lie on a path $w \, \dots \, Q$ that is extendable to a red path $w \, \dots \, \alpha$. 
    \begin{itemize}
        \item There is a red path $w \, \dots \, Q$ that is extendable to a red path $w \, \dots \, \alpha$, where $w \, \dots \, Q$ does not containing any of $q_b', q_x', q_y'$.
        Since one of them is on every red path $w \, \dots \, \alpha$, we get that $q_y'$ is on every extension $Q \, \dots \, \alpha$.
        We get that $\gamma \, S \, Q \, \dots \, q_y' \, \dots \, \alpha \, \beta \, b' \, x'$ contains an induced $P_8$. 
        \item There is a red path $w \, \dots \, Q$ that is extendable to a red path $w \, \dots \, \alpha$ and contains only~$q_y'$. Then, $\gamma \, w \, \dots \, q_y' \, \dots \, Q \, \dots \,  \alpha \, \beta \, b' \, x'$ contains an induced $P_8$.
        \item There is a red path $w \, \dots \, Q$ that is extendable to a red path $w \, \dots \, \alpha$ and contains only~$q_x'$. 
        Recall that $q_x'$ has no neighbour in $Q$. Thus, $\gamma \, w \, \dots \, q_x' \, \dots \, Q \, \dots \, \alpha \, \beta \, b'$ contains an induced $P_8$. Note that this still holds if $q_y'$ lies on $Q \, \dots \, \alpha$, since then $q_y'q_x' \notin E$.
        \item There is a red path $w \, \dots \, Q$ that is extendable to a red path $w \, \dots \, \alpha$ and contains exactly $q_x'$ and $ q_y'$. Then, $\gamma \, w \, \dots \, \{q_x', q_y'\} \, \dots \, Q \, \dots \, \alpha \, \beta \, b'$ contains an induced $P_8$. Note that $q_x'$ and $q_y'$ may appear in any order and may have extra vertices between them.
        \item Otherwise, $q_b'$ lies on every red path $w \, \dots \, Q$ that is extendable to a red path $w \, \dots \, \alpha$. If there is such a path that does not contain $q_b$, then $\gamma \, w \, \dots \, q_b' \, \dots \, Q  \, \dots\, \alpha \, \beta \, \dots\, b$ contains an induced~$P_8$. To see this recall that $q_b'$ has no neighbour in $Q$, so there is at least one additional vertex.
        Else, $q_b$ is contained in every red path $w \, \dots \, Q$ that is extendable to a red path $w \, \dots \, \alpha$. Then, $\gamma \, w \, \dots \, q_b' \, \dots \, q_b  \, b \, \dots\, \beta \, \alpha$  contains an induced $P_8$. To see this, note that $q_b\alpha \notin E$, else $x$ would have been coloured blue before we chose it as an uncoloured vertex.
        Note that in both cases, wherever we add $q_x'$ and $q_y'$ on the path, we still find an induced path on at least~$8$ vertices.
    \end{itemize}
\noindent
Hence, in all cases we find an induced $P_8$, a contradiction to the assumption that $G$ is $P_8$-free. It follows, that the uncoloured neighbour $\gamma$ of $S$ cannot exist. Thus, $C_1$ only contains vertices without a red neighbour. Recall that by R4, every uncoloured component has a red neighbour. 
Let $v \in C_1$ and consider a path through its uncoloured component to a red neighbour of the component. 
This path necessarily contains another uncoloured vertex, since $v$ has no red neighbour. 
Let $w$ be the vertex on the path that is adjacent to a red vertex. Then, $w \in C_2$, a contradiction to the fact that every vertex in $C_2$ has no uncoloured neighbour. It follows that $C_1$ is empty.
Hence, all uncoloured vertices are contained in $C_2$ and thus form an independent set. We apply Lemma~\ref{L-Indep Set}. We either obtain a minimum red-blue $(R,B)$-colouring of $G$ and remember its value, or conclude that no such colouring exists and discard the branch. This ends the second side branch.

\smallskip\noindent
Thus, we first check whether colouring $y'$ blue leads to a valid red-blue $(R,B)$-colouring of~$G$ in which case we remember its value. This can be done in polynomial time, since the number of branches is polynomial and the application of Lemma~\ref{L-Indep Set} runs in polynomial time.

\smallskip\noindent
\textit{We colour $y'$ red and propagate the colouring.} The number of uncoloured vertices decreases and we repeat the process from the point where we checked whether all connected components of $G[Z]$ have size~$1$ (after having coloured $y$). Whenever we obtain a valid red-blue colouring, we remember its value. Whenever we discard a branch, we consider the next.
In every iteration of choosing $y'$, the number of uncoloured vertices decreases. We only recurse in the case where $y'$ is red and thus at most $O(n)$ times. For each step in the recursion, we consider $O(n^3)$ sidebranches in the case where $y'$ is blue, resulting in $O(n^4)$ branches.

Hence, if we colour $y$ blue, we can find in polynomial time a minimum red-blue $(R,B)$-colouring of $G$ or conclude that no such colouring exists and thus, the claim follows.
\end{claimproof}

\noindent
We first check whether colouring $y$ blue leads to a valid red-blue $(R,B)$-colouring of $G$ in which case we remember its value. Note that this can be done in polynomial time.
Then, we colour $y$ red and propagate the colouring. The number of uncoloured vertices decreases and we repeat the process from the point where we checked whether all connected components of $G[Z]$ have size~$1$. Whenever we obtain a valid red-blue colouring, we remember its value. Whenever we discard a branch, we consider the next.
In every iteration of choosing $y$, the number of uncoloured vertices decreases. We only recurse in the case where $y$ is red and thus at most $O(n)$ times. For each step in the recursion, we consider $O(n^5\cdot n^{15}\cdot n^4)$ sidebranches in the case where $y$ is blue, resulting in $O(n^{26})$ branches.

\medskip \noindent
The correctness of our algorithm follows from its description. If we found a valid red-blue colouring, we output the smallest value of any valid red-blue colouring we obtained, otherwise we return that no such colouring exists. Since in both cases our algorithm runs in polynomial time, the total running time is polynomial.
\end{proof}

\section{Hardness Results}
\label{sec:hardness}

For our hardness results, we reduce from \textsc{Vertex Cover}.
A \emph{vertex cover} is a set $S$ of vertices such that every edge is incident to a vertex of $S$. In the problem \textsc{Vertex Cover}, we are given an instance $(G,k)$ and ask whether $G$ has a vertex cover of size~$k$.
The problem \textsc{Vertex Cover} is well known to be \NP-complete, see e.g.~\cite{Ka72}.

\begin{theorem}\label{t-3p3}
\minmc{} is \NP-hard for $3P_3$-free graphs of radius~$2$ and diameter~$3$.
\end{theorem}
\begin{proof}
Let $(G,k)$ be an instance of \textsc{Vertex Cover}. We construct a graph $G'$ from~$G$, consisting of vertex gadgets, edge gadgets, and cover gadgets (which will ensure that every edge of $G$ is covered) as follows.
For each vertex $v \in V$ we construct a \emph{vertex gadget} consisting of two cliques $C_v$ and $C_v'$,  where $C_v$ is of size $|E| + 2$ and $C_v'$ of size $\max(3,\degree(v) +1)$, and two vertices $u_v$ and $u_v'$. We pick a vertex $w$ in $C_v$ and $w'$ in $C_v'$ and add the edges $u_vw, wu_v', u_vw', w'u_v'$ to the gadget, see Figure~\ref{fig:3p3_constr}.
\begin{figure}[t]
    \centering
    \begin{tikzpicture}
    \tikzset{
        rahmen/.style={
            rounded corners = 5pt,
            draw,
            dashed
        }
    }

    \newcommand{\vdotss}[2]{%
        \node at ($(#1)!0.5!(#2)$) {\hspace{0.7pt}\rotatebox{90}{$\cdot$\hspace{-1pt}$\cdot$\hspace{-1pt}$\cdot$}}%
    }
    \newcommand{\fframe}[3]{%
        \draw[rahmen,#1] ($(#2.north west)+(-0.4,0.4)$) rectangle ($(#3.south east)+(0.4,-0.4)$)%
    }
    \newcommand{\rclique}[2]{%
        \expandafter\node[vertex,draw=#1,fill=#1] (#2 k) at (0,1.5) {};
        \expandafter\node[vertex,draw=#1,fill=#1] (#2 2) at (0,0.5) {};
        \node[vertex,draw=#1,fill=#1] (#2 1) at (0,0) {};
        \vdotss{#2 2}{#2 k};
    }
    \newcommand{\clique}[2]{%
        \rclique{#1}{#2};
        \fframe{#1}{#2 k}{#2 1};
    }
    \newcommand{\iclique}[2]{%
        \expandafter\node[vertex,draw=#1,fill=#1] (#2 k) at (0,0) {};
        \expandafter\node[vertex,draw=#1,fill=#1] (#2 2) at (0,1) {};
        \node[vertex,draw=#1,fill=#1] (#2 1) at (0,1.5) {};
        \vdotss{#2 2}{#2 k};
        \fframe{#1}{#2 1}{#2 k};
    }

    \begin{scope}[xscale=0.5,yscale = 0.9, rotate = 90]
    \rclique{nicered}{u};
    \begin{scope}[xshift=2cm,yshift=1cm]
        \clique{lightblue}{cu};
        \node[](t) at (0.6,1){$C_{v_1}$};
    \end{scope}
    \begin{scope}[xshift=2cm,yshift=-2.5cm]
        \iclique{nicered}{cpu};
        \node[](t) at (0.6,0.4){$C_{v_1}'$};
    \end{scope}
    \node[bvertex] (up) at (4,0) {};

    \begin{scope}[yshift=-7cm]
        \rclique{nicered}{v};
        \begin{scope}[xshift=2cm,yshift=1cm]
            \clique{nicered}{cv};
            \node[](t) at (0.6,1){$C_{v_2}$};
        \end{scope}
        \begin{scope}[xshift=2cm,yshift=-2.5cm]
            \iclique{lightblue}{cpv};
            \node[](t) at (0.6,0.4){$C_{v_2}'$};
        \end{scope}
        \node[bvertex] (vp) at (4,0) {};
    \end{scope}

    \node[rvertex] (e) at ($(u 1)!0.5!(v k)$) {};

    \node[rvertex, label = left:$a$] (a) at (0, -10.5) {};
    \node[bvertex, label = left: $b$] (b) at (4, -10.5) {};
    
    \fframe{nicered}{u k}{a};
    \node[](t) at (0,-11.7){$C$};
    \fframe{lightblue}{up}{b};
     \node[](t) at (4,-11.7){$C'$};
    
    \draw[tedge] (a) -- (b);

    \draw[edge] (e) -- (cpu 2);
    \draw[tedge] (e) -- (cpv 2);

    \draw[tedge] (u 1) -- (cu 1);
    \draw[tedge] (u 2) -- (cu 2);
    \draw[tedge] (u k) -- (cu k);
    \draw[edge] (cu 1) -- (up);
    \draw[edge]  (u 1) -- (cpu 1);
    \draw[tedge] (cpu 1) -- (up);

    \draw[edge] (v 1) -- (cv 1);
    \draw[edge] (v 2) -- (cv 2);
    \draw[edge] (v k) -- (cv k);
    \draw[tedge] (cv 1) -- (vp);
    \draw[tedge]  (v 1) -- (cpv 1);
    \draw[edge] (cpv 1) -- (vp);
    \end{scope}

\end{tikzpicture}
    \caption{
    The graph $G'$ constructed for a two adjacent vertices $v_1$ and $v_2$. 
    The given colouring indicates that $v_1$ is in the vertex cover $S$ while $v_2$ is not.
    }
    \label{fig:3p3_constr}
\end{figure}
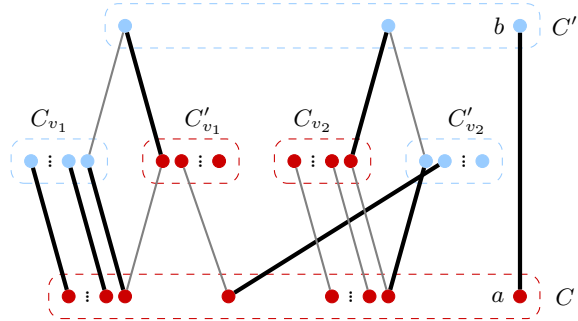
We let $C$ and $C'$ be cliques with vertex sets $\{u_v, v \in V\}$ and $\{u_v', v \in V\}$, respectively. That is, we connect the vertices $u_v$ and $u_v'$ of the variable gadgets to two cliques $C$ and $C'$.

Consider an edge $e = vv'$ in $E$.  We construct an \emph{edge gadget} as follows.
Pick a vertex $w \in C_{v}'$ and a vertex $w' \in C_{v'}'$, such that $w$ and $w'$ do not have a neighbour outside of $C_v'$ and $C_{v'}'$.
Add a new vertex $e$ to $C$ and add the edges $ew, ew'$.

For each $v \in V$, we add $|E| + 1$ vertices to $C$. We connect these vertices to $|E|+1$ vertices in $C_v$ such that each of these vertices in $C_v$ has exactly one neighbour among the additional vertices in $C$ and vice versa. We call this the \emph{cover gadget} of $v$.

As a last step, we add vertices $a$ to $C$ and $b$ to $C'$ together with the edge $ab$.
Note that when adding vertices to $C, C'$ and $C_v$, for $v \in V$, we add edges such that these sets remain cliques.

We first show that $G'$ is $3P_3$-free and has radius~$2$ and diameter~$3$.
To see that $G'$ is $3P_3$-free, note first that $G' - (C+C')$ is a disjoint union of cliques and thus contains no induced path consisting of more than $2$ vertices. Hence, every induced $P_3$ in $G'$ contains at least one vertex of one of the cliques $C$ and $C'$. Since no two induced paths can contain a vertex in the same clique, we conclude that $G'$ is $3P_3$-free.
The distance from $a$ to any vertex in $C$ and to $b$ is one. Since every vertex in the cliques $C_v$ and $C_v'$ for any $v \in V$ is adjacent to a vertex in $C$, their distance to $a$ is two. Further, since the distance of any vertex in $C'$ to~$b$ is one, the distance to $a$ is two and thus $G'$ has radius~$2$.
The distance between any two vertices in $C$ and $C'$ is at most~$3$, which also holds for any two vertices in $G'- (C + C')$. Thus, the diameter of $G'$ is at most~$3$.

\medskip
\noindent
We claim that $G$ has a vertex cover of size $k$ if and only if $G'$ has a matching cut of size $\mu$ satisfying
\begin{equation}
\label{eq:p3free}
2|V| + (1+|E|)k + 1 \leq \mu \leq 2|V| + (1 + |E|)k + |E| + 1
\end{equation}
Note that by Observation~\ref{basic-observations} this is the case if and only if $G'$ has a red-blue colouring of value~$\mu$.

First, suppose that $G$ has a vertex cover $S$ of size $k$. We construct a red-blue colouring as follows.
We colour $C$ red and $C'$ blue.
Let $v \in V$. If $v \in S$ then colour $C_v$ blue and $C_v'$ red, otherwise colour $C_v$ red and $C_v'$ blue.
To see that this colouring is valid, note first that every vertex has at most two neighbours outside of the clique in which it is contained.
Let $v \in C$ be a vertex contained in a vertex gadget. Then, since $C_v$ and $C_v'$ are coloured differently, $v$ has at most one blue neighbour. By the same argument, a vertex $v' \in C'$ contained in a vertex gadget has at most one red neighbour.
Every vertex in $C_v$ and $C_v'$ has at most one neighbour in $C$ and one in $C'$, so regardless of its colour it has at most one neighbour of the other colour.
Consider a vertex $e \in C$, corresponding to an edge $e = uv$ of~$G$. Since $S$ is a vertex cover, at least one of $u$ and $v$ is contained in $S$ and thus at least one of $C_u'$ and $C_v'$ is red. 
All other vertices in $C$ or $C'$ have at most one neighbour outside of their clique. 

We compute the value of the colouring. Every vertex gadget has two bichromatic edges, so the vertex gadgets contribute a value of $2|V|$.
Every edge gadget contains $0$ or $1$ bichromatic edge, thus their contribution is between $0$ and $|E|$.
For the cover gadgets, if a vertex $v \in V$ is contained in $S$, then the corresponding cover gadget contributes a value of $|E + 1|$ and otherwise contributes a value of $0$.
Finally, the edge $ab$ contributes one to the value of the colouring.
In total, the value $\mu$ of the red-blue colouring satisfies Equation~(\ref{eq:p3free}).

\medskip \noindent
We now consider a fixed red-blue colouring of $G'$ of value $\mu$ satisfying Equation~(\ref{eq:p3free}) for some $k\in\mathbb N$.
Note first that the cliques $C$, $C'$, and $C_v$ and $C_v'$ for all $v \in V$ are monochromatic in any valid red-blue colouring of $G'$.
\begin{myclaim}\label{c:3p3:cliquesdiffer1}
The cliques $C$ and $C'$ have different colours in every valid red-blue colouring of~$G'$. Similarly, for any $v \in V$, the cliques $C_v$ and $C_v'$ have different colours in every valid red-blue colouring of $G'$.
\end{myclaim}
\begin{claimproof}
We first show that $C$ and $C'$ have different colours.
Suppose for a contradiction that they have the same colour, say $C$ and $C'$ are both red.
Then for every $v \in V$ the vertices in $C_v$ and $C_v'$ that are adjacent to both $C$ and $C'$ have each two red neighbours and are thus red.
Hence, the cliques $C_v$ and $C_v'$ are all red, and thus all vertices of $G'$ are red, a contradiction.

We now show that for any $v \in V$, the cliques $C_v$ and $C_v'$ have different colours.
Suppose that both are coloured the same, say both are red. Then the blue vertex $v' \in C'$, which is adjacent to one vertex of each of $C_v$ and $C_v'$, has two red neighbours, a contradiction.
\end{claimproof}

\smallskip\noindent
Thus, each vertex gadget contributes a value of $2$ in every valid red-blue colouring of $G'$ .
Further, every edge gadget contributes a value of at most~$1$, since its vertex in $C$ has at most one neighbour of the other colour.
Also, the edge $ab$ is bichromatic and thus, contributes a value of $1$.
Without loss of generality and due to Claim~\ref{c:3p3:cliquesdiffer1} we may assume that $C$ is red and $C'$ is blue.
We define a set $S$ containing all vertices $v \in V$ such that $C_v$ is coloured blue.
Note that a cover gadget for some $v \in V$ contributes $(1 + |E|)$ bichromatic edges if and only if the clique $C_v$ is blue, that is, if and only if $v \in S$.
From Equation~(\ref{eq:p3free}) we immediately get $(1 + |E|)k \leq \mu - 2|V| - 1 \leq (1 + |E|)k + |E|$ and thus there are 
$\left\lfloor \frac{\mu - 2|V| - 1}{1 + |E|} \right\rfloor = k$ vertices in $S$.

It remains to show that $S$ is a vertex cover of $G$.
Consider an edge $uv \in E$. $S$ is a vertex cover if at least one of $u$ and $v$ is in $S$. 
That is, at least one of $C_u$ and $C_v$ has to be blue.
Suppose that this is not the case, that is, both $C_u$ and $C_v$ are red and thus, by Claim~\ref{c:3p3:cliquesdiffer1} both $C_u'$ and $C_v'$ are blue. Since they are contained in an edge gadget and thus have a common red neighbour in $C$, this contradicts the validity of the colouring.
\end{proof}

\noindent
If we replace the cliques by complete bipartite graphs, we obtain a similar result for bipartite graphs. However, the radius and diameter increase by~$1$ and the construction is not $3P_3$-free.

\begin{restatable}{theorem}{thmhardnessbipartite}
    \label{t-biprad3}
\minmc{} is \NP-hard for bipartite graphs of radius~$3$ and diameter~$4$.
\end{restatable}
\begin{proof}
    Let $(H,k)$ be an instance of \textsc{Vertex Cover}.
    Let $G$ be the graph constructed for this instance in the proof of Theorem~\ref{t-3p3}.
    We construct a bipartite graph $G'$ with partition classes $A$ and $B$ as follows. Replace every clique $C$ of size~$k$ by a complete bipartite graph $C'$ with partition classes each of size~$k$.
    We say that a vertex $u \in V(C)$ corresponds to two vertices in $V(C')$, $u_a \in A\cap V(C')$ and $u_b \in B\cap V(C')$.
    Let $u, v\in V$ and let $u_a,u_b, v_a$ and $v_b$ be the corresponding vertices.
    If $uv \in E$ we add the edges $u_av_b$ and $u_bv_a$.

    Thus, $G'$ is clearly bipartite since each edge connects a vertex in $A$ with a vertex in $B$.
    To see that $G'$ has radius~$3$ and diameter~$4$ note first that this is an increase of~$1$ compared to the radius and diameter of $G$.
    Consider a path in $G$ which is used to show that the distance between $u$ and $v$ in $G$ is bounded.
    Note that the distance between $u_a$ and one of $v_a$ and $v_b$ equals $\dist(u,v)$.
    Since $v_a$ and $v_b$ are adjacent, we get that $\dist(u_a,v_a) \leq \dist(u,v) +1$ and $ \dist(u_a,v_b) \leq \dist(u,v) +1$.
    By symmetry, the same holds for $u_b$ and thus, $G'$ has radius~$3$ and diameter~$4$.

    We claim that $H$ has a vertex cover of size $k$ if and only if $G'$ has a matching cut of size~$\mu' = 2\mu$ where
    \[
2|V| + (1+|E|)k + 1 \leq \mu \leq 2|V| + (1 + |E|)k + |E| + 1
\]
as in the proof of Theorem~\ref{t-3p3}.
Note that by Observation~\ref{basic-observations} this is the case if and only if $G'$ has a red-blue colouring of value~$\mu'$.

First, suppose that $H$ has a vertex cover $S$ of size $k$.
We construct a red-blue colouring of $G$ of value~$\mu$ using the arguments from Theorem~\ref{t-3p3}.
For every vertex $u\in V$, we colour $u_a$ and $u_b$ in $V(G')$ with the same colour as~$u$.
This immediately leads to a valid red-blue colouring of $G'$ of value $2\mu = \mu'$ since every bichromatic edge $uv$ in $G$ connects two vertices $u$ and $v$ which do not belong to the same clique and each such edge corresponds to two edges $u_av_b$ and $u_bv_a$ in~$G'$.

Consider now a fixed red-blue colouring of $G'$ of value~$\mu'$.
Then, two vertices $u_a$, $u_b$ corresponding to the same vertex $u$ in $G$ are coloured the same since they are contained in a complete bipartite graph with partition classes of size at least~$3$.
Thus, the red-blue colouring of $G$ where a vertex $u$ has the same colour as $u_a$ in $G'$ is a red-blue colouring of $G$ of value~$\mu$.
Using the arguments from Theorem~\ref{t-3p3} we obtain a vertex cover $S$ of $H$ of size $k$.
\end{proof}

\section{Conclusion}
\label{sec:conc}

Combining our results with results from the literature, we obtain the following partial complexity classifications for \minmc{} and \mc.

\begin{theorem}[\cite{Ch84,FLPR25}]\label{t-h-dicho}
    For a graph $H$, \minmc{} on $H$-free graphs is 
    \begin{itemize}
        \item polynomial-time solvable if $H \subseteq_i sP_2 + S_{1,1,3}$, $sP_2 + P_8$, or $sP_2 + P_6 + P_4$, for some $s \geq 0$, 
        \item \NP-hard if $H \supseteq_i 3P_3, K_{1,4}, C_r, r\geq 3$, or $H_i^*, i \geq 1$.
    \end{itemize}
\end{theorem}

\begin{theorem}[\cite{Ch84,FLPR25,LL26,LPR22,LPR23a,Mo89}]
For a graph $H$, \mc{} on $H$-free graphs is 
\begin{itemize}
    \item polynomial-time solvable if $H \subseteq_i sP_3 + S_{1,1,3}$, $sP_3 + P_6 + P_4$ or $sP_3 + P_8$, for some $s \geq 0$, 
    \item NP-complete if $H \supseteq_i K_{1,4} , P_{14} , 2P_7 , 3P_5 , C_r$, $r \geq 3$, or $H_i^*$, $i \geq 1$.
\end{itemize} 
\end{theorem}

\noindent
In both cases, the computational complexity remains open only for a constant number of graphs $H$, where every connected component of $H$ is either a path or a subdivided claw. For \minmc{}, $31$ cases remain open when $H$ is connected. Among them are the cases of $P_9$ and $P_{10}$-free graphs. Solving these would close the gap for $P_r$-free graphs. An extension of our techniques would require a third level of domination. Dealing with this seems challenging. Even if for example Lemma~\ref{L-monodom} can be generalised, it remains challenging to get to the situation where the coloured part of the graph dominates the uncoloured part. Hence, new ideas are needed. For \mc{}, even more cases remain open. The most interesting among them is probably the case of $P_r$-free graphs for $9 \leq r \leq 13$. 

For (bipartite) graphs of bounded radius and diameter, we can combine Theorems~\ref{T-P6},~\ref{t-3p3} and~\ref{t-biprad3} to obtain the following dichotomy for \minmc. However, for \mc{} the complexity on bipartite graphs of radius~$3$ remains open~(\cite{Lu25}).

\begin{theorem}
    \minmc{} is polynomial-time solvable for 
    \begin{itemize}
        \item graphs of radius~$r$, for $r \leq 1$, and diameter~$r$, for $r \leq 2$, and
        \item bipartite graphs of radius~$r$, for $ r \leq 2$, and diameter~$r$, for $r \leq 3$.
    \end{itemize}
    \minmc{} is \NP-complete for 
    \begin{itemize}
        \item graphs of radius~$r$, for $r \leq 2$, and diameter~$r$, for $r \leq 3$, and
        \item bipartite graphs of radius~$r$, for $ r \leq 3$, and diameter~$r$, for $r \leq 4$.
    \end{itemize}
\end{theorem}

\noindent
As a last problem, note that \maxmc{} is \NP-hard for graphs of maximum degree~$3$~(\cite{LPR24}), while every graph with at least~$7$ vertices and maximum degree~$3$ has a matching cut (\cite{Mo89}). It would thus be interesting to determine the complexity of \minmc{} on graphs of maximum degree~$3$.

 \bibliography{ref}

\end{document}